\documentclass[a4paper,twoside]{amsart}
\usepackage[utf8]{inputenc}
\usepackage[T1]{fontenc}
\usepackage[british]{babel}
\usepackage{amssymb,amsmath,amsthm,amscd,amsfonts,graphics}
\usepackage{mathrsfs}
\usepackage{enumerate}
\usepackage{verbatim}
\usepackage{lmodern}
\usepackage[all]{xy}
\usepackage{xcolor}
\usepackage{enumerate}

\theoremstyle{definition}
\newtheorem{thm}{Theorem}[section]
\newtheorem{cor}[thm]{Corollary}
\newtheorem{lem}[thm]{Lemma}
\newtheorem{prop}[thm]{Proposition}

\newtheorem{defn}[thm]{Definition}

\theoremstyle{remark}
\newtheorem{rem}[thm]{Remark}
\numberwithin{equation}{section}

\newcommand{\R}{\mathbb{R}}
\newcommand{\N}{\mathbb{N}}

\newcommand{\C}{\mathbb{C}}

\newcommand{\fA}{\mathcal{A}}
\newcommand{\fB}{\mathcal{B}}
\newcommand{\fC}{\mathcal{C}}

\newcommand{\fL}{\mathcal{L}}
\newcommand{\Lc}{\mathrm{Ly}}

\newcommand{\om}{\omega}

\newcommand{\sg} {\sigma}

\newcommand{\supp}{\text{supp}}

\newcommand{\pf}{\textnormal{PF}}

\newcommand{\la}{\langle}
\newcommand{\ra}{\rangle}

\renewcommand{\subset}{\subseteq}

\author[Anderson-Sackaney]{Benjamin Anderson-Sackaney}
\address{Department of Mathematics and Statistics, University of Victoria, Victoria, British Columbia, V8P 5C2, Canada}
\email{bandersonsackaney@uvic.ca}
\thanks{BA-S was supported by PIMS and the Simons Foundation - Simons grant PPTW GR023618.}

\author[de Laat]{Tim de Laat}
\address{Department of Mathematics, Kiel University, Heinrich-Hecht-Platz 6, 24118 Kiel, Germany}
\email{delaat@math.uni-kiel.de}
\thanks{TdL was supported by the Deutsche Forschungsgemeinschaft under Germany’s Excellence Strategy -- EXC 2044 -- 390685587, Mathematics M\"unster: Dynamics -- Geometry -- Structure, and the Deutsche Forschungsgemeinschaft -- Project-ID 427320536 -- SFB 1442.}

\author[Samei]{Ebrahim Samei}
\address{Department of Mathematics and Statistics, University of Saskatchewan, Saskatoon, Saskatchewan, S7N 5E6, Canada}
\email{ebrahim.samei@usask.ca}
\thanks{ES was partially supported by NSERC Discovery Grant RGPIN-2025-04833.}

\author[Wiersma]{Matthew Wiersma}
\address{Department of Mathematics and Statistics, University of Winnipeg, 515 Portage Avenue, Winnipeg, Manitoba, Canada  R3B 2E9}
\email{m.wiersma@uwinnipeg.ca}
\thanks{MW was partially supported by NSERC Discovery Grant RGPIN-2022-03062.}

\title[Entropies and Poisson boundaries of random walks]
{Entropies and Poisson boundaries of random walks on groups with rapid decay}

\begin{document}

\begin{abstract}
    Let $G$ be a countable group and $\mu$ a probability measure on $G$. We build a new framework to compute asymptotic quantities associated with the $\mu$-random walk on $G$, using methods from harmonic analysis on groups and Banach space theory, most notably complex interpolation. It is shown that under mild conditions, the Lyapunov exponent of the $\mu$-random walk with respect to a weight $\om$ on $G$ can be computed in terms of the asymptotic behavior of the spectral radius of $\mu$ in an ascending class of weighted group algebras, and we prove that for natural choices of $\om$ and $\mu$, the Lyapunov exponent vanishes. Also, we show that the Avez entropy of the $\mu$-random walk can be realized as the Lyapunov exponent of $\mu$ with respect to a suitable weight.
    
    We apply our results to stationary dynamical systems consisting of an action of a group with the property of rapid decay on a \hyphenation{pro-ba-bi-li-ty}probability space. We prove that whenever the associated Koopman representation is weakly contained in the left-regular representation of the group, then the Avez entropy coincides with the Furstenberg entropy of the stationary space. This gives a characterization of (Zimmer) amenability for actions of rapid decay groups on stationary spaces.

    Next, by considering the spectral radius in the algebras of $p$-pseudofunctions on $G$, we introduce a new asymptotic quantity, which we call convolution entropy. We show that for groups with the property of rapid decay, the convolution entropy coincides with the Avez entropy.

\end{abstract}

\maketitle

\section{Introduction} \label{sec:introduction}
Let $G$ be a countable group with identity element $e$, and let $\mu$ be a probability measure on $G$. Let $(Y_n)_{n \in \N}$ be a sequence of $G$-valued independent identically distributed random variables with law $\mu$. Let $W_0 := e$, and set $W_n := Y_1\ldots Y_n$, i.e.~$W_n$ is distributed with law $\mu^{*n}$. The random process $(W_n)_{n \in \N_0}$ is called the $\mu$-random walk on $G$ (with starting point $e$).

Asymptotic quantities and properties associated with random walks on a group $G$ reflect important information on the structure of $G$. For instance, for finitely generated groups, Kesten showed in his seminal work that given a symmetric probability measure $\mu$ supported on a finite generating set of $G$, the norm of the Markov operator $M : \ell^2(G) \to \ell^2(G)$ associated with the $\mu$-random walk, which is defined by $Mf(s) = \sum_{t \in G} f(st)\mu(t)$, equals $1$ if and only if $G$ is amenable \cite{Kesten}.

One of the most prominent asymptotic quantities associated with the $\mu$-random walk on $G$ is its (Avez) entropy, which is defined as
\begin{align*}
h(G,\mu) := - \lim_{n \rightarrow \infty} \frac{1}{n} \sum_{s \in G} \mu^{*n} (s) \, \log \mu^{*n} (s).
\end{align*}
It is a fundamental result due to Avez \cite{Avez}, Derriennic \cite{Derriennic} and Kaimanovich--Vershik \cite{Kaim-Vers 1} that if $\mu$ is a non-degenerate probability measure satisfying some mild conditions, the entropy of the $\mu$-random walk vanishes if and only if all bounded $\mu$-harmonic functions on $G$ are constant. A function $f : G \to \R$ is $\mu$-harmonic if $Mf(s)=\sum_{t \in G} f(st)\mu(t)=f(s)$. The property that every bounded $\mu$-harmonic function is constant coincides with the triviality of the Poisson boundary of the $\mu$-random walk, which is a measure space associated with the random walk that reflects structural and asymptotic properties of $G$ (see Section \ref{subsec:poissonboundary}). For background on random walks on groups, we refer to \cite{Lyons-Peres,woess,Zheng}.

Another important asymptotic quantity associated with a $\mu$-random walk is its Lyapunov exponent, which is defined with respect to a weight on $G$. Under mild conditions on the weight $\om$, the Lyapunov exponent of the $\mu$-random walk with respect to $\om$ is defined as
\begin{align*}
\Lc_\om(G,\mu):=\displaystyle\lim_{n\rightarrow \infty}\dfrac{1}{n}\sum_{s\in G}\mu^{\ast n}(s) \, \log \omega (s).
\end{align*}
This notion of Lyapunov exponent goes back to \cite{Nevo1}. In more general settings, Lyapunov exponents and spectra play a pivotal role in the areas of dynamical systems and random processes. They have a rich history and there are deep connections between them and random walks on groups; see e.g.~\cite{Furm1,Wilkinson2017}.

The first aim of this article is to build a natural framework that gives new ways to compute the Lyapunov exponent and the Avez entropy, leading to fundamental relations between them. This involves methods from harmonic analysis on groups and functional analysis. In particular, we use techniques from Banach space theory, most notably complex interpolation. Our first concrete result is Theorem \ref{T:Lyap expo-weighted L1}, which gives an explicit way to compute the Lyapunov exponent of the $\mu$-random walk with respect to a weight $\om$ in terms of the asymptotic behavior of the spectral radius of $\mu$ in an ascending class of weighted group algebras. Then, relying on a classical result from harmonic analysis and potential theory, namely the theorem of de la Vall\'ee Poussin, we derive a vanishing result for the Lyapunov exponent in Theorem \ref{T:vanishing Lyap expo-log of lenght function} for a specific class of weights and natural conditions on $\mu$ that will be of importance later. Next, we define a weighted version of the Avez entropy (see Definition \ref{D:weighted Avez entropy}), which we denote by $h_\om(G,\mu)$, whose computation relies heavily on complex interpolation results. Subsequently, in Theorem \ref{T:Lyap Expo dominates entropy-inverse integrable weights}, we prove that the Avez entropy of the $\mu$-random walk can be realized as the Lyapunov exponent of $\mu$ with respect to a suitable weight.

In Section \ref{sec:stationarydynamicalsystems}, we apply our framework to stationary dynamical systems. Let $(G,\mu)$ be as before, and let $(X,\xi)$ be a standard probability space such that $G$ acts measurably on $(X,\xi)$. The measure $\xi$ is $\mu$-stationary if $\mu \ast \xi = \xi$. Then $(X,\xi)$ is called a $(G,\mu)$-stationary space. Both measure-preserving actions as well as continuous boundary actions on compact spaces are examples of stationary actions. This class of actions has been studied extensively; see \cite{Furs-Glas 1} for an overview.

Analytic properties of a dynamical system $G \curvearrowright (X,\xi)$ can be studied through the associated Koopman representation $\pi_X$ of $G$ and the asymptotic behavior of matrix coefficients of $\pi_X$. Using our framework, this leads to a formula for the Furstenberg entropy $h_\mu(X,\xi)$ with respect to $(G,\mu)$ whenever $G$ has the property of rapid decay (shortly, property RD), which is an analytic regularity property shared among several classes of very different types of groups (see Section \ref{subsec: groups with RD} for the definition of this property as well as some general background and a discussion of examples). The following is our first main result.

\begin{thm}[Theorem \ref{T:Square integrable-inverse Harich Chandra function}] \label{T:Square integrable-inverse Harich Chandra function introduction}
Let $G$ be a countable group satisfying property RD with respect to a length function $\mathcal{L}$, and let $\mu$ be a non-degenerate probability measure on $G$ with both finite entropy and finite $\log (1+\fL)$-moment. Let $(X,\xi)$ be a $(G,\mu)$-stationary space such that the associated Koopman representation $\pi_X$ is weakly contained in the left-regular representation $\lambda_G$. Then
\begin{align*}
    h_\mu(X,\xi)=h(G,\mu)=-2\lim_{n\rightarrow \infty}\dfrac{1}{n}\sum_{s\in G}\mu^{\ast n}(s) \log \Xi_X(s).
\end{align*}
\end{thm}
As a consequence, for groups with property RD, we can characterize (Zimmer) amenability of the action of $G$ on $(X,\xi)$.
\begin{thm}[Theorem \ref{T:charc amen action-RD groups}] \label{T:charc amen action-RD groups introduction}
Let $G$ be a countable group satisfying property RD with respect to a length function  $\fL$, let $\mu$ be a non-degenerate probability measure on $G$ with both finite entropy and finite $\log(1+\fL)$-moment, and let $(X,\xi)$ be a $(G,\mu)$-stationary space. Then the following are equivalent:
\begin{enumerate}[(i)]
    \item The space $(X,\xi)$ is a measure-preserving extension of the Poisson boundary of $(G,\mu)$.
    \item The action $G \curvearrowright (X,\xi)$ is amenable (in the sense of Zimmer).
    \item The representation $\pi_{X}$ is weakly contained in $\lambda_G$.
    \item $h_\mu(X,\xi)=h(G,\mu)$.
\end{enumerate}
\end{thm}

In particular, if $(X,\xi)$ is a $\mu$-boundary, then $\pi_X$ is weakly contained in $\lambda_G$ if and only if $(X,\xi)$ coincides with the Furstenberg-Poisson boundary of $(G,\mu)$ (see Corollary \ref{C:charc entopry-RD groups}).

We would like to point out the relation between Theorem \ref{T:charc amen action-RD groups introduction} and the existence of a unique Poisson boundary. For a non-degenerate probability measure $\mu$ on $G$, we say that $(G,\mu)$ has a \textbf{unique Poisson boundary realization} if the Poisson boundary of $(G,\mu)$ has a compact model realization with the unique $\mu$-stationary Borel probability measure (see \cite[Definition 3.9]{Hart-Kal 2}). Several natural examples of Poisson boundaries have this property, e.g.~free groups, hyperbolic groups, certain $\mathrm{CAT}(0)$-groups, small cancellation groups and mapping class groups. This is generally achieved through applying a powerful method developed by Kaimanovich in \cite{Kaim 1}, known today as \emph{Kaimanovich's strip criterion}, to give a concrete realization of the Poisson boundary and show its uniqueness; we refer to \cite[Section 3.2]{Hart-Kal 2} for further details. It is shown in \cite[Theorem 9.2]{nevosageev} that whenever $(G,\mu)$ has the unique Poisson boundary realization, every amenable $(G,\mu)$-space is a measure-preserving extension of the Poisson boundary of $(G,\mu)$. In particular, the equivalence of (i) and (ii)  in Theorem \ref{T:charc amen action-RD groups introduction} holds in these situations. Since \cite[Theorem 4.4]{Furs-Glas 1} implies the equivalence of (i) and (iv) in Theorem \ref{T:charc amen action-RD groups introduction} and Zimmer amenability implies the weak containment of the Koopman representation in the left-regular representation of $G$ \cite[Theorem 4.3.1]{Ana-Dela 1}, the crucial part in Theorem \ref{T:charc amen action-RD groups introduction} is to establish that (iii) implies (iv). To the best knowledge of the authors, this is new even for free groups. This is particularly interesting because there is a large overlap between groups with a unique Poisson boundary realization and groups with property RD. The latter property is used in a fundamental way to prove Theorem \ref{T:Square integrable-inverse Harich Chandra function introduction}, but it is not clear whether one could use the former property to obtain the same result.

In the final section (Section \ref{S:convolution entropy}), we introduce a new asymptotic quantity associated with a $\mu$-random walk. Instead of looking at the asymptotic behaviour of the spectral radius of $\mu$ in weighted $L^p$-spaces, we consider the spectral radius in the space $\mathrm{PF}_p(G)$ of $p$-pseudofunctions on $G$, which is defined as the norm closure in $B(\ell^p(G))$ of the image of the left-regular representation $\lambda_p : \ell^1(G) \to B(\ell^p(G))$. The crucial difference here comes from the fact that the space of $p$-pseudofunctions captures information about the group law directly, through the left-regular representation.
\begin{defn} \label{d:convolutionentropy}
Let $\mu$ be a probability measure on $G$ with finite Shannon entropy (see Section \ref{subsec:entropyrw} for the definition). The \textbf{convolution entropy} $c(G,\mu)$ of the $\mu$-random walk on $G$ is defined as
\begin{align*}
    c(G,\mu):=\displaystyle\lim_{p\rightarrow \infty}-p\log r_{\pf_q(G)}(\mu),
\end{align*}
where $q$ depends on $p$ through $\dfrac{1}{p}+\dfrac{1}{q}=1$.
\end{defn}
In Theorem \ref{T:Avez entropy-p conv operators-I}, we show that in general, $c(G,\mu)$ is dominated by the Avez entropy $h(G,\mu)$. On the other hand, as indicated before, the use of our framework becomes particularly appealing for groups with property RD, as we show that such groups, the convolution entropy coincides with the Avez entropy.

\begin{thm}[Theorem \ref{T:Avez entropy-p conv operators-RD}] \label{T:Avez entropy-p conv operators-RD intro}
Let $G$ be a countable group satisfying property RD with respect to a length function $\fL$, and let $\mu$ be a non-degenerate probability measure on $G$ with both finite entropy and finite $\alpha$-moment with respect to $1+\fL$ for some $\alpha>0$. Then the Avez entropy and the convolution entropy of the $\mu$-random walk coincide:
\begin{align*} 
 h(G,\mu)=c(G,\mu).
\end{align*}
\end{thm}

\section*{Acknowledgements}
The third-named author is grateful to Vern Paulsen for a nice conversation and also for pointing out the similarity between the weighted Shannon entropy and the relative entropy/Kullback--Leibler divergence. This is summarized in Remark \ref{R:weighted Shannon entropy vs relative entropy}. We thank the referee for their careful reading and for several useful suggestions that significantly improved the exposition of the article. 

\section{Preliminaries}

\subsection{Entropy of random walks} \label{subsec:entropyrw}

Let $G$ be a countable group, let $e$ denote its identity element, and let $\mu$ be a probability measure on $G$. The measure $\mu$ is called symmetric if $\mu(s^{-1}) = \mu(s)$ for all $s \in G$ and non-degenerate if its support generates $G$ as a semigroup.

We first recall the notion of the $\mu$-random walk on $G$, following the conventions of \cite{Zheng}. Let $(Y_n)_{n \in \N}$ be a sequence of $G$-valued independent identically distributed random variables with law $\mu$. Set $W_0 := e$, and define $W_n := Y_1\ldots Y_n$. Then $W_n$ is distributed with law
\[
	\mu^{*n} =\underbrace{\mu \ast \mu \ast \ldots \ast \mu}_{n-\text{times}},
\]
where the convolution product $\mu \ast \nu$ of two probability measures $\mu$ and $\nu$ is given by $\mu \ast \nu (s)=\sum_{t \in G} \mu(t)\nu(t^{-1}s) = \sum_{t \in G} \mu(st^{-1})\nu(t)$. The random process $(W_n)_{n \in \N_0}$ is called the $\mu$\textbf{-random walk} on $G$ (with starting point $e$).

Equivalently, the $\mu$-random walk on $G$ can be viewed as the Markov chain with state space $G$, transition probabilities $p(s,t)=\mu(s^{-1}t)$, and initial distribution $\delta_e$. The space of paths of the $\mu$-random walk is the set $\mathcal{Z} = G^{\N}$ equipped with the Markovian measure $\mathbb{P}_{\mu}$ given by
\[
    \mathbb{P}_{\mu} \left( \{ z \in \mathcal{Z} : z_1=s_1, \, z_2=z_1s_2, \, \ldots, \, z_n=z_{n-1}s_n \} \right) = \mu(s_1) \ldots \mu(s_n)
\]
for $n \in \N$ and $s_1,\ldots,s_n \in G$. Note that $\mathbb{P}_{\mu}(\{ z \in \mathcal{Z} : z_n = s \})=\mu^{\ast n}(s)$.

The \textbf{Shannon entropy} of the random variable $W_n$ (which only depends on the law $\mu^{*n}$) is defined as
\[
	H(G,\mu^{*n}) = -\sum_{s\in G} \mu^{*n} (s) \, \log \mu^{*n} (s),
\]
where we use the convention that $\mu(s)\log \mu(s)=0$ whenever $s\notin \supp \, \mu$. If $H(G,\mu) < \infty$, then $H(G,\mu^{*n}) < \infty$ for all $n$. More precisely, the mapping
\begin{align*}
  \N \to [0,\infty), \; n \mapsto H(G,\mu^{*n})
\end{align*}
is subadditive on $\N$, so that the following limit exists:
\begin{align*}
h(G,\mu) := \lim_{n\rightarrow \infty} \dfrac{H(G,\mu^{\ast n})}{n} =\inf \left\{ \dfrac{H(G,\mu^{\ast n})}{n}:n\in \mathbb{N}\right\}.
\end{align*}
The limit $h(G,\mu)$ is called the \textbf{(Avez) entropy} of the $\mu$-random walk.

Let $\fL : G\to [0,\infty)$ be a subadditive function, i.e.~$\fL(st)\leq \fL(s)+\fL(t)$ for all $s,t\in G$, and $\alpha>0$. (Note that in particular length functions on groups are subadditive.) The measure $\mu$ is said to have \textbf{finite $\alpha$-moment} with respect to $\fL$ if $\sum_{s\in G} \mu(s)\fL(s)^\alpha<\infty$.
For $\alpha=1$, we say that $\mu$ has finite first moment with respect to $\fL$. In particular, if $G$ is finitely generated with finite symmetric generating set $S$ containing $e$ and associated word-length function $|\cdot|_S : G \to \N_0$, we say that $\mu$ has finite $\alpha$-moment (with respect to $|\cdot|_S$) if $\sum_{s} \mu(s)|s|^\alpha_S<\infty$.

The \textbf{speed} of the $\mu$-random walk on a finitely generated group with finite symmetric generating set $S$ containing $e$ is defined as
\[
  l(G,\mu)=\lim_{n\to \infty} \frac{L(G,\mu^{*n})}{n}, \ \text{where}\ \ L(G,\mu^{*n}):=\sum_{s\in G} \mu^{*n}(s) \, |s|_S.
\]
We suppress the dependency on the generating set in the notation of the speed, because it is usually clear which generating set is used. If $\mu$ has finite first moment, then the $\mu$-random walk has finite speed and the following inequality holds:
\begin{align}\label{Eq:Fund inequality}
  h(G,\mu) \leq v_S \, l(G,\mu),
\end{align}
where $v_S = \lim_{n\to \infty} \frac{1}{n} \log |S^{n}|$ is the logarithmic \textbf{volume growth} of $G$ with respect to $S$.
For background and details on entropy of random walks, we refer to \cite{Lyons-Peres,woess,Zheng}.

\subsection{Poisson boundary} \label{subsec:poissonboundary}

The \textbf{(Furstenberg--)Poisson boundary} of a $\mu$-random walk on $G$ is a probability space that reflects aspects of the asymptotic behavior of the random walk. This notion goes back to Furstenberg \cite{Furstenberg1963}. We only give one characterization of the Poisson boundary \cite[Theorem 2.11]{badershalom} (see also \cite[Theorem 2.2]{boutonnethoudayer}). Let $(X,\xi)$ be a $(G,\mu)$-stationary space (i.e.~$(X,\xi)$ is a standard probability space equipped with a measurable $G$-action such that $\mu \ast \xi = \xi$), let $H^{\infty}(G,\mu)$ denote the space of all bounded $\mu$-harmonic functions, and consider the map
\[
    T : L^{\infty}(X,\xi) \to H^{\infty}(G,\mu), \; f \mapsto Tf,
\]
where
\[
    Tf(s)=\int_X f(sx) d\xi(x).
\]
The Poisson boundary $(\Pi_{\mu},\nu_{\infty})$ is the unique (up to $G$-equivariant measure space isomorphism) $(G,\mu)$-stationary space $(X,\xi)$ for which the map $T$ is a $G$-equivariant isometric isomorphism.

For a thorough introduction and an overview of the Poisson boundary, including different characterizations, we refer to \cite[Section 2]{badershalom} and \cite{Furm1,Furstenberg1973,Kaim-Vers 1}.

\subsection{$L^p$-spaces and complex interpolation}\label{S:complex interpolation}

We assume that the reader is familiar with the basics of (complex) interpolation spaces, as covered in e.g.~\cite{BL}. This background is only explicitly used to prove Proposition \ref{prop:interpolationconsequence} in this section.

Let $(X,\nu)$ be a $\sg$-finite measure space, and let $\om : X \to (0,\infty)$ be a positive measurable function on $X$. The space $L^p(X,\om)$ is defined to be the $L^p$-space
\[
L^p(X,\omega) := L^p(X,\om^p \nu) = \{f : X \to \C \textrm{ measurable} : f\omega \in L^p(X,\nu)\}
\]
equipped with the norm
\[
    \|f\|_{p,\om} = \left( \int_X |f|^p \om^p d\nu \right)^{\frac{1}{p}}.
\]
By \cite[Theorem 5.5.3]{BL}, for $1\leq p_0 \leq  p_1 < \infty$, measurable functions $\om_0$ and $\om_1$ on $X$, and $\theta \in (0,1)$, we have the complex interpolation space
\begin{equation}\label{Eq:interpolation-weigh lp}
L^p(X,\omega)=(L^{p_0}(X,\omega_0),L^{p_1}(X,\omega_1))_\theta,
\end{equation}
where $p$ and $\om$ are given by
\begin{equation}\label{Eq:interpolation-weigh lp-relations}
\frac{1}{p}=\frac{1-\theta}{p_0}+\frac{\theta}{p_1}, \qquad \om=\om_0^{1-\theta} \om_1^{\theta}.
\end{equation}
Moreover, for every $f\in L^{p_0}(X,\omega_0)\cap L^{p_1}(X,\omega_1)$, we have
\begin{equation}\label{Eq:interpolation-weigh lp norm-relations}
\|f\|_{p,\omega} \leq \|f\|^{1-\theta}_{p_0,\omega_0}\|f\|^\theta_{p_1,\omega_1}.
\end{equation}
Similarly, by \cite[Theorems 5.5.3 and 5.1.2]{BL}, we have
\begin{equation*} 
(L^{p_0}(X,\omega_0),L^{\infty}(X,\nu))_\theta=L^p(X,\omega),
\end{equation*}
where
\begin{equation*} 
\frac{1}{p}=\frac{1-\theta}{p_0}, \qquad \om = \om_0^{1-\theta}.
\end{equation*}

The following proposition follows easily from \eqref{Eq:interpolation-weigh lp}, \eqref{Eq:interpolation-weigh lp-relations} and \eqref{Eq:interpolation-weigh lp norm-relations} and is used frequently throughout this article.

\begin{prop}  \label{prop:interpolationconsequence}
Let $(X,\nu)$ be a $\sg$-finite measure space, let $\omega$ be a positive measurable function on $X$, and let $1< q_0\leq p_0< \infty$ with $\dfrac{1}{p_0}+\dfrac{1}{q_0}=1$. Then for every
$f\in L^1(X,\nu)\cap L^{q_0}(X,\omega^{1/p_0})$
with $\Vert f\Vert_{1}=1$, the mapping
\begin{equation*}
[p_0,\infty) \rightarrow \R, \ \ p \mapsto -p \, \log \Vert f\Vert_{q,\omega^{1/p}},
\end{equation*}
where $1< q \leq q_0 \leq p_0 \leq p < \infty$ with $\dfrac{1}{p}+\dfrac{1}{q}=1$, is an increasing function. Equivalently, the mapping
\begin{equation*}
(1,q_0] \rightarrow \R, \ \ q \mapsto -p \, \log \Vert f\Vert_{q,\omega^{1/p}}
\end{equation*}
is a decreasing function.
\end{prop}

\begin{proof}
Note that $q_0 \leq 2 \leq p_0$, and let $1 < q < r \leq q_0$ and $p_0 \leq u < p < \infty$ be such that
$$\dfrac{1}{p}+ \dfrac{1}{q}=1 \quad \text{and}\quad \dfrac{1}{u}+ \dfrac{1}{r}=1.$$
Let $\theta \in (0,1)$ be such that
$$\dfrac{1}{q}=1-\theta +\dfrac{\theta}{r}.$$ Equivalently, $\theta =\dfrac{u}{p}$.

Therefore, if we put $\omega_0=1$ and $\omega_1=\omega^{1/u}$, by \eqref{Eq:interpolation-weigh lp}, \eqref{Eq:interpolation-weigh lp-relations} and \eqref{Eq:interpolation-weigh lp norm-relations}, we obtain
$$\Vert f\Vert_{q,\omega^{1/p}} \leq \Vert f\Vert_{1}^{1-\theta} \Vert f\Vert_{r,\omega^{1/u}}^{\theta} = \Vert f\Vert_{r,\omega^{1/u}}^{\frac{u}{p}}.$$
Hence,
$$\Vert f\Vert_{q,\omega^{1/p}}^p \leq \Vert f\Vert_{r,\omega^{1/u}}^{u},$$
so that
$$-u \, \log \Vert f\Vert_{r,\omega^{1/u}} \leq -p \, \log \Vert f\Vert_{q,\omega^{1/p}}.$$
This completes the proof.
\end{proof}

We also use the following result on interpolation of $n$-linear maps; see e.g.~{\cite[Section 10.1]{Cal}}.
\begin{prop} \label{T:n linear-interpol}
	Let $(\fA_i,\fB_i)$, with $i=1,\ldots, n+1$, be interpolation pairs, and suppose that
	$$T_\fA : \fA_1\times \ldots \times \fA_n \to \fA_{n+1} \quad \textrm{and} \quad T_\fB : \fB_1\times \ldots \times \fB_n \to \fB_{n+1}$$
	are bounded $n$-linear maps that coincide on $(\fA_1\cap\fB_1)\times \ldots \times (\fA_n\cap \fB_n)$.
Let
$$T : (\fA_1\cap\fB_1)\times \ldots \times (\fA_n\cap \fB_n)\to \fA_{n+1}\cap \fB_{n+1}$$ be the restriction of $T_\fA$ (or $T_\fB$) to $(\fA_1\cap \fB_1)\times \ldots \times (\fA_n\cap \fB_n)$. Fix $\theta\in (0,1)$, and let $\fC_{i}=(\fA_i,\fB_i)_\theta$, $i=1,\ldots,n+1$, be the interpolation of $\fA_i$ and $\fB_i$ with parameter $\theta$. Then $T$ extends to a bounded $n$-linear map
	\begin{equation*} 
	T_\fC : \fC_1\times \ldots \times \fC_n \to \fC_{n+1}
	\end{equation*}
	such that
	\begin{equation*} 
	\|T_\fC\|\leq \|T_\fA\|^{1-\theta}\|T_\fB\|^\theta.
	\end{equation*}
\end{prop}

\subsection{Weighted group algebras and $\ell^p$-spaces} \label{subsec:weightedgroupalgebras}

\begin{defn}\label{D:weight}
A \textbf{weight} on a countable group $G$ is a function $\omega : G \rightarrow [a,\infty)$, with $a>0$, that is weakly submultiplicative, i.e.~there exists $C>0$ such that for all $s,t\in G$,
\begin{align}\label{Eq:weight relations}
\om(st)\leq C \om(s)\om(t).
\end{align}
Equivalently, $\om$ is a weight on $G$ if $\log \om$ is weakly subadditive, i.e.~there exists $D \in \R$ such that for every $s,t\in G$,
\begin{equation*}
\log \om(st)\leq \log \om(s) + \log \om(t) + D.
\end{equation*}
Two weights $\om$ and $\om'$ on $G$ are said to be equivalent if there are $c,c'>0$ such that
$$c\,\om \leq \tilde{\om} \leq c'\,\om.$$
\end{defn}

\begin{rem}\label{R:growth of weights}
\
\begin{enumerate}[(i)]
    \item Let $\om$ be a weight on $G$, and let $C>0$ be such that \eqref{Eq:weight relations} holds.
    By multiplying with $C$, we see that $\om$ is equivalent to the weight $C\om$, which is submultiplicative. Also, we can always rescale a weight to an equivalent weight with values in $[1,\infty)$.
    \item If $G$ is finitely generated with finite symmetric generating set $S$ containing $e$ and $\om$ is a weight on $G$, then the function
\begin{align*}
  n \mapsto \sup_{s\in S^n} \om(s)=\sup_{s_1,s_2,\ldots,s_n \in S} \om(s_1 \ldots s_n)
\end{align*}
is weakly submultiplicative, so that the limit
\begin{align*}
\om_S:=\lim_{n\to \infty}  \sup_{s\in S^n} \om(s)^{1/n}
\end{align*}
exists; this limit $\om_S$ is called the {\textbf{growth rate}} of $\om$. Equivalent weights have the same growth rate.
\end{enumerate}
\end{rem}

The following type of weight is of pivotal importance in this article.

First, recall that a length function on a group $G$ is a function $\mathcal{L} : G \to [0,\infty)$ satisfying $\mathcal{L}(e)=0$, $\mathcal{L}(s^{-1})=\mathcal{L}(s)$ and $\mathcal{L}(st) \leq \mathcal{L}(s) + \mathcal{L}(t)$ for all $s,t \in G$.
\begin{defn}\label{D:polynomial weights}
Let $\fL : G\to [0,\infty)$ be a length function on $G$ (see Section \ref{subsec: groups with RD} for the definition). For $d \geq 1$, the function
\[
    \mathcal{P}_{\fL}^{d} : s \mapsto (1+\fL(s))^d
\]
defines a weight on $G$; it is called the \textbf{polynomial weight} of degree $d$ with respect to $\fL$.
\end{defn}

For a weight $\om$ on $G$, the space
$$\ell^1(G,\om):=\{f : f\om\in \ell^1(G)\},$$
together with the norm $\|f\|_{1,\om}:=\sum_{s\in G} |f(s)| \, \om(s)$,
becomes a Banach algebra with respect to pointwise addition and convolution; it is called the \textbf{weighted $\ell^1$-algebra} (also known as the \textbf{Beurling algebra}) on $G$ with respect to the weight $\om$ (see \cite[Section 3.7]{RS}). For every $p\geq 1$, we define
\begin{align} \label{eq:omegap}
  \om_p(s)=\om(s)^{\frac{1}{p}}, \qquad s \in G.
\end{align}
Then $\om_p$ is also a weight on $G$. Moreover, for every $1\leq u\leq p$, we have the following complex interpolation relations (see Section \ref{S:complex interpolation}):
\begin{align}\label{Eq:complex interpolation-weight L1}
  \ell^1(G,\om_p)=(\ell^1(G),\ell^1(G,\om_u))_\theta, \;\; \text{with} \;\; \theta=\frac{u}{p},
\end{align}
and
\begin{align}\label{Eq:complex interpolation-weight L1-norm relations}
  \|f\|_{1,\om_p}\leq \|f\|_1^{1-\theta}\|f\|_{1,\om_u}^\theta \;\; \textrm{for all} \;\; f\in \ell^1(G)\cap \ell^1(G,\om_u).
\end{align}

Moreover, for the weighted $\ell^q$-spaces associated with a weight on the group $G$, we have the following: Suppose that $1 < q_{1}<q_0 < \infty $ and that $p_0$ and $p_1$ are the conjugates of $q_0$ and $q_1$, respectively. Then, by putting $\omega_0=1$ and $\omega_{p_0}=\omega^{1/p_0}$ (as in the equation \eqref{eq:omegap}), we have $\om_{p_1}=\om_{p_0}^\theta$, where $\theta=\frac{p_0}{p_1}$. Thus, by applying equations \eqref{Eq:interpolation-weigh lp}, \eqref{Eq:interpolation-weigh lp-relations}, and \eqref{Eq:interpolation-weigh lp norm-relations}, we have the following complex interpolation relations (see Section \ref{S:complex interpolation}):
\begin{align}\label{Eq:complex interpolation-weight Lq}
  \ell^{q_1}(G,\om_{p_1})=(\ell^1(G),\ell^{q_0}(G,\om_{p_0}))_\theta, \quad \text{with}\ \theta=\frac{p_0}{p_1},
\end{align}
and
\begin{align}\label{Eq:complex interpolation-weight Lq-norm relations}
  \|f\|_{q_1,\om_{p_1}}\leq \|f\|_1^{1-\theta}\|f\|_{q_0,\om_{p_0}}^\theta, \quad f\in \ell^1(G)\cap \ell^{q_0}(G,\om_{p_0}).
\end{align}

\subsection{Groups with rapid decay}\label{subsec: groups with RD}

\begin{defn} \label{D:RD}
    A countable group $G$ has the \textbf{property of rapid decay} (also called \textbf{property RD}) with respect to a length function (see Section \ref{subsec:weightedgroupalgebras} for the definition) $\mathcal{L}$ if there exists $d \in \N$ such that the canonical inclusion $\C[G] \hookrightarrow C^*_{r}(G)$ extends to a bounded linear map $\ell^2(G,\om) \to C^*_{r}(G)$, where $C^*_{r}(G)$ is the reduced $C^*$-algebra of $G$ and $\om$ is the weight on $G$ given by $\om(s)=(1+\mathcal{L}(s))^d$.
\end{defn}
For finitely generated groups, it is natural to consider the word length in Definition \ref{D:RD}. Since any two length functions on a finitely generated group are equivalent, it moreover suffices to consider the word length and it makes sense to speak about property RD without specifying a length function; see \cite[Section 2]{Chatt 1}.

The property of rapid decay was first established for the free groups $\mathbb{F}_n$, for $n \geq 2$, in the seminal work of Haagerup \cite{Haagerup}, and its systematic study goes back to \cite{Joli}, in which the property was proved for groups with polynomial growth and for classical hyperbolic groups. Nowadays, many more classes of groups with property RD, mostly of geometric nature, are known, and we refer to \cite{Chatt 1,Garncarek} for overviews. Unlike several other approximation properties for infinite groups, such as amenability, the Haagerup property, and weak amenability, property RD is a phenomenon that occurs both in lower (Lie theoretic) rank, i.e.~rank $\leq 1$, and in higher rank, i.e.~rank $\geq 2$. The examples mentioned so far are either amenable or rank $1$. However, it was proved by Ramagge, Robertson and Steger that groups acting properly and cocompactly on $\widetilde{A}_2$-buildings (e.g.~torsion-free lattices in $\mathrm{PGL}_3(\mathbb{Q}_p)$), which are of rank $2$, satisfy property RD. Subsequently, Lafforgue proved that cocompact lattices in $\mathrm{SL}(3,\mathbb{R})$ and $\mathrm{SL}(3,\mathbb{C})$ have property RD \cite{Lafforgue}, which was extended to cocompact lattices in $\mathrm{SL}(3,\mathbb{H})$ and in the exceptional group $\mathrm{E}_{6(-26)}$ in \cite{Chatterji2}. A deep conjecture due to Valette states that cocompact lattices in arbitrary semisimple groups have property RD \cite[Conjecture 7]{Valette}. Cocompactness cannot be dropped in this conjecture, since it is known to fail for non-cocompact lattices, as follows from an important result of Lubotzky, Mozes and Raghunathan \cite{LMR}. (The fact that $\mathrm{SL}(n,\mathbb{Z})$, with $n \geq 3$, does not have property RD was already shown in \cite{Joli}.)

Property RD has led to striking applications, most notably in the direction of the Baum-Connes conjecture. We refer to \cite{Chatt 1} for an overview, including several equivalent definitions of the property.

\section{Lyapunov exponent of random walks}

As before, $G$ denotes a countable discrete group.

\begin{defn}\label{D:Lyap expo}
Let $\om$ be a weight on $G$. For every probability measure $\mu$ with finite first $\log \om$-moment (that is, $\sum_{s \in G} \mu(s) \log \om (s) < \infty$), the \textbf{Lyapunov exponent} of $\mu$ with respect to $\om$ is defined as
\begin{align*}
\Lc_\om(G,\mu):=\displaystyle\lim_{n\rightarrow \infty}\dfrac{1}{n}\sum_{s\in G}\mu^{\ast n}(s) \log \omega (s).
\end{align*}
\end{defn}

To the knowledge of the authors, the explicit definition of this version of the Lyapunov exponent for random walks on groups (in this setting and generality) goes back to the work of Nevo \cite{Nevo1}. In fact, Nevo's definition is slightly more general than ours and covers random walks on non-discrete groups as well. Nevo attributes this notion of Lyapunov exponent to Margulis \cite[pp.~189-190]{Marg1}, who uses it implicitly and who in turn refers to \cite{G1,Ledrappier1982} for other uses of this quantity.

The following lemma shows that the above definition makes sense. The argument is straightforward and given for completeness.

\begin{lem} \label{L:subadditive-mu log omega}
Let $\om$ be a weight on $G$, and let $\mu$ be a probability measure on $G$ with finite first $\log \om$-moment. Then
$$\displaystyle\lim_{n\rightarrow \infty}\dfrac{1}{n}\sum_{s\in G}\mu^{\ast n}(s) \log \omega (s)$$
exists.
\end{lem}

\begin{proof}[Sketch of proof]
For a probability measure $\mu$ on $G$, set $S(\mu):=\displaystyle\sum_{s\in G}\mu(s) \log \omega (s)$. First, suppose that $\om$ is submultiplicative. A standard computation, applying properties of the logarithm and the submultiplicativity of $\om$, shows that for every two probability measures $\mu$ and $\nu$ on $G$ with finite $\log \om$-moment, we have
$S(\mu\ast \nu) \leq S(\mu)+S(\nu)$.
Hence, the mapping $n\mapsto S(\mu^{\ast n})$ is subadditive, so that the limit
$$\displaystyle\lim_{n\rightarrow \infty}\dfrac{1}{n} \, S(\mu^{\ast n})=\lim_{n\rightarrow \infty}\dfrac{1}{n}\sum_{s\in G}\mu^{\ast n}(s) \log \omega (s)$$
exists. The lemma now follows, since every weight is equivalent to a \hyphenation{sub-mul-ti-pli-ca-tive}submultiplicative weight (see Remark \ref{R:growth of weights}) and two equivalent weights yield the same limit.
\end{proof}

\begin{prop}\label{P:Lyap. expo-Basic properties}
  Let $\om$ be a weight on $G$, and let $\mu$ be a probability measure on $G$ with finite first $\log \om$-moment. Then the following hold:
  \begin{enumerate}[(i)]
  \item The Lyapunov exponent $\Lc_{\om}(G,\mu)$ is a non-negative number.
  \item If $\tilde{\om}$ is a weight on $G$ equivalent to $\om$, then $\Lc_{\tilde{\om}}(G,\mu)=\Lc_\om(G,\mu)$.
  \item If $G$ is finitely generated with finite symmetric generating set $S$ containing $e$ and $\mu$ has finite first moment with respect to $|\cdot|_S$, then the Lyapunov exponent of $\mu$ with respect to $\om$ exists, and
  \[
    \Lc_\om(G,\mu)\leq \left( \log \om_S \right) \, l(G,\mu),
  \]
  where $\om_S$ is the growth rate of $\om$ with respect to $S$ (see Remark \ref{R:growth of weights}).
  \end{enumerate}
\end{prop}

\begin{proof}
(i): Let $a > 0$ be such that $\om(s) \geq a$ for all $s \in G$. Then $\sum_{s\in G} \mu(s) \log \om(s) \geq \sum_{s\in G} \mu(s) \log a=\log a$.
By replacing $\mu$ with $\mu^{*n}$ and taking the limit, we obtain:
$$\Lc_\om(G,\mu) \geq \lim_{n\rightarrow \infty}\dfrac{1}{n}\sum_{s\in G}\mu^{\ast n}(s) \log a \geq \lim_{n\rightarrow \infty}\dfrac{\log a}{n} =0.$$
(ii): This straightforward computation is left to the reader.\\
(iii): Without loss of generality, we assume that $\om\geq 1$. Let $S$ be a finite symmetric generating set for $G$ containing $e$, and let $\om_S$ be
the growth rate of $\om$ with respect to $S$. Fix $C>\om_S$. Then there exists $k_0\in \N$ such that for every $k > k_0$, we have $\sup_{s\in S^k} \om(s)\leq C^k$. Hence,
\begin{align*}
  \sum_{s\in G} \mu(s) \log \om(s) & \leq \sum_{s\in S^{k_0}} \mu(s) \log \om(s) \; + \sum_{k=k_0+1}^{\infty} \; \sum_{|s|=k} \mu(s) \log \om(s) \\
  & \leq \sum_{s\in S^{k_0}} \mu(s) \log \om(s) \; +\sum_{k=k_0+1}^{\infty} \; \sum_{|s|=k} \mu(s) \log C^k \\
  & = \sum_{s\in S^{k_0}} \mu(s) \log \om(s) \; + \; \log C \sum_{k=k_0+1}^{\infty} \; \sum_{|s|=k} \mu(s) \; |s| \\
  & = \sum_{s\in S^{k_0}} \mu(s) \log \om(s) \; + \; \left( \log C \right) \; L(G,\mu).
\end{align*}
By replacing $\mu$ with $\mu^{*n}$ and taking the limit, we have
\begin{align*}
  \Lc_\om(G,\mu) &= \lim_{n\rightarrow \infty}\dfrac{1}{n}\sum_{s\in G}\mu^{\ast n}(s) \log \omega (s)\\
   & \leq \limsup_{n\rightarrow \infty} \; \frac{1}{n}\sum_{s\in S^{k_0}} \mu^{*n}(s) \log \om(s) \; + \; \left( \log C \right) \; \lim_{n\rightarrow \infty} \frac{1}{n} \; L(G,\mu^{*n}) \\
   & \leq  \left( \sup_{s\in S^{k_0}}\om(s) \right) \limsup_{n\rightarrow \infty} \; \frac{1}{n}\sum_{s\in S^{k_0}} \mu^{*n}(s) \; + \; \left( \log C \right) \; l(G,\mu) \\
   &= \left( \log C \right) \; l(G,\mu).
   \end{align*}
The result follows by letting $C\to \om_S^+$.
\end{proof}

As we saw in Definition \ref{D:Lyap expo}, the integrability of $\log \om$ with respect to the probability measure $\mu$ ensures the existence of the Lyapunov exponent of $\mu$. In the following, we show that if we consider a slightly stronger integrability assumption, we obtain an alternative formula for $\Lc_\om(G,\mu)$ in terms of the asymptotic behavior of the spectral radius of $\mu$ in an ascending class of weighted group algebras. First, we prove the following result, which is used several times throughout the paper.

\begin{lem} \label{R:derivative log-power}
Let $(X,\xi)$ be a probability space, and let $f$ be a non-negative measurable function on $X$ such that the integral $\displaystyle \int_X \log f(x) d\xi(x)$ exists and the set $\{x\in X: f(x)=0 \}$ has $\xi$-measure zero.
Suppose that for some $\alpha_0>0$, the function $f^{\alpha_0}$ lies in $L^1(X,\xi)$. Then

\begin{align} \label{eq:derivative log-power}
 \lim_{\alpha \to 0^+} \frac{1}{\alpha} \log \int_X f(x)^\alpha d\xi(x)=\int_X \log f(x)d\xi(x).
\end{align}
\end{lem}

\begin{proof}
First, observe that for every $\alpha \geq 0$ and $t>0$,
\begin{align*}
    t^\alpha \left| \log t \right| \leq \max\{-\log t,0\}+t^{\alpha} \max\{\log t,0\}.
\end{align*}
Hence, if we choose $C>0$ such that
\begin{align} \label{Eq:log-power inequality}
    \log t \leq C t^{\alpha_0/2} \textrm{ for all } t\in [1,\infty),
\end{align}
then for every $0\leq \alpha\leq \alpha_0/2$ and $t>0$, we have
\begin{align*}
    t^\alpha \left| \log t \right| \leq \max\{-\log t,0\}+Ct^{\alpha_0}.
\end{align*}
Therefore, if $f(x) > 0$, we have
\begin{align*}
   \left| \frac{d[f(x)^\alpha]}{d\alpha} \right| & = f(x)^\alpha \left| \log f(x) \right| \\
    &\leq \max \{-\log f(x),0 \}+C f(x)^{\alpha_0}
\end{align*}
for every $0\leq \alpha\leq \alpha_0/2$. Moreover, if we put $Y=f^{-1}(0,1]$, then
\begin{align*}
& \int_X \Bigl( \max \{-\log f(x),0 \}+C f(x)^{\alpha_0} \Bigr) d\xi(x) \\
&=\int_Y -\log f(x) d\xi(x) +C\int_{X} f(x)^{\alpha_0} d\xi(x) \\
&=\int_X -\log f(x) d\xi(x) + \int_{X \setminus Y} \log f(x) d\xi(x) +C\int_{X} f(x)^{\alpha_0} d\xi(x)  \\
    & \leq -\int_X \log f(x) d\xi(x) +2C\int_{X} f(x)^{\alpha_0} d\xi(x) <\infty. 
\end{align*}
Hence, by differentiation under the integral sign (\cite[2.27 Theorem]{Folland}), it follows that the mapping
\begin{align*}
    F \colon \alpha \mapsto \int_X f(x)^\alpha d\xi(x)
\end{align*}
is differentiable on $[0,\alpha_0/2]$, with
\begin{align*}
    F'(\alpha) =
   \int_X f(x)^\alpha \log f(x) d\xi(x),  \quad \alpha \in [0,\alpha_0/2].
\end{align*}
By l'H\^opital's rule, we obtain
\begin{align*}
 \lim_{\alpha \to 0^+} \frac{1}{\alpha} \log \displaystyle\int_X f(x)^\alpha d\xi(x) &= \frac{ \int_X f(x)^\alpha \log f(x) d\xi(x)}{\int_X f(x)^\alpha d\xi(x)}\bigg|_{\alpha=0}\\
 &=\int_X \log f(x)d\xi(x).
\end{align*}
\end{proof}

\begin{thm}\label{T:Lyap expo-weighted L1}
Let $\om$ be a weight on $G$, and let $\mu$ be a probability measure on $G$ with $\mu \in \displaystyle\bigcup_{p\geq 1}\ell^1(G,\omega_p)$. Then $\mu$ has a finite Lyapunov exponent with respect to $\om$, and we have
\begin{align}\label{Eq:Lyap expo-sepc radius weight L1}
\Lc_\om(G,\mu):=\lim_{p \rightarrow \infty} p\log r_{\ell^1(G,\omega_p)}(\mu),
\end{align}
where $r_{\ell^1(G,\omega_p)}(\mu)$ is the spectral radius of $\mu$ in the Banach algebra $\ell^1(G,\omega_p)$.
\end{thm}

\begin{proof}
Without loss of generality, we assume that $\om$ is submultiplicative and $\om\geq 1$. Fix $p_0\geq 1$ such that $\mu\in \ell^1(G,\om_{p_0})$, and let $p_0\leq p$. Then, by \eqref{Eq:complex interpolation-weight L1}, $\mu\in \ell^1(G,\om_{p})$. Also, if $p_0 \leq u < p$ and $\theta =\dfrac{u}{p}$, then we have
$$\omega_p=(\omega_u)^{\theta},$$
so that by \eqref{Eq:complex interpolation-weight L1-norm relations}, we obtain
\begin{align*}
    1 \leq \Vert \mu \Vert_{1,\om_p} \leq \Vert \mu\Vert_{1}^{1-\theta} \Vert \mu\Vert_{1,\omega_u}^{\theta}  =\Vert \mu\Vert_{1,\omega_u}^{\frac{u}{p}}.
\end{align*}
This implies that the function
$$[p_0,\infty )\rightarrow [0,\infty ), \;\; p \mapsto p\log \Vert \mu\Vert_{1,\omega_p}$$
is decreasing, and hence the limit
$$F(G,\mu):=\displaystyle\lim_{p \rightarrow \infty} p\log \Vert \mu\Vert_{1,\omega_p}$$
exists.
Note that $\sum_{s\in G}\mu (s)\log \omega (s) < \infty$, as $\log \om \leq C \om^{1/p_0}$ for some $C>0$. This implies that $\mu$ has finite Lyapunov exponent with respect to $\om$. Moreover, by Lemma \ref{R:derivative log-power}, with $(X,\xi)=(G,\mu)$ and $f=\om$, we obtain
\begin{align*}
F(G,\mu) =\displaystyle\lim_{\alpha \rightarrow 0^{+}}\dfrac{1}{\alpha}\log \displaystyle\sum_{s\in G} \mu (s) \omega (s)^{\alpha} 
= \displaystyle\sum_{s\in G}\mu (s)\log \omega (s).
\end{align*}

Furthermore, by replacing $\mu$ with $\mu^{*n}$ and taking the limit, we obtain
\begin{align*} 
\Lc_\om(G,\mu)=\lim_{n\rightarrow \infty}\dfrac{1}{n} \, F(G,\mu^{\ast n})=\inf_{n\in \N}\dfrac{1}{n}\,F(G,\mu^{\ast n}),
\end{align*}
where the last equality holds because the function $\N\to [0,\infty), \; n\mapsto F(G, \mu^{\ast n})$ is subadditive; cf.~Lemma \ref{L:subadditive-mu log omega}.
Furthermore, since for every $n\in \N$, the function
$$[p_0,\infty )\rightarrow [0,\infty ), \;\; p \mapsto p\log \Vert \mu^{*n}\Vert_{1,\omega_p}$$
is decreasing, we obtain
\begin{align*}
 \Lc_\om(G, \mu) &= \inf_{n\in \N}\dfrac{1}{n} \, F(G,\mu^{\ast n}) =\inf_{n\in \N} \inf_{p \geq p_0} \frac{p}{n}\log \Vert \mu^{*n}\Vert_{1,\omega_p} \\ &= \inf_{p \geq p_0} \inf_{n\in \N} \frac{p}{n} \log \Vert \mu^{*n}\Vert_{1,\omega_p} = \inf_{p \geq p_0}  p\log r_{\ell^1(G,\omega_p)}(\mu) \\ &= \lim_{p \rightarrow \infty} p\log r_{\ell^1(G,\omega_p)}(\mu).
\end{align*}

\end{proof}

\subsection{Vanishing of the Lyapunov exponent} \label{subsec:vanishinglyapunovexponents}

As mentioned in Section \ref{sec:introduction}, the vanishing of the Avez entropy of a $\mu$-random walk on a countable group $G$ is equivalent to the triviality of the Poisson boundary of $(G,\mu)$. This significant result in probabilistic group theory provides a powerful computational tool for determining the triviality of the Poisson boundary.

In this section, we investigate a similar phenomenon, namely the vanishing of the Lyapunov exponent. We restrict ourselves to the case where the weight comes from a subadditive function, although looking at more general settings is interesting as well and will be considered in future work. Explicitly, we show that in this case, the Lyapunov exponent always vanishes. This is the main result of this section and will provide a key tool in our investigation of Poisson boundaries.

Our proof relies on a major classical result in harmonic analysis and potential theory, namely the theorem of de la Vall\'{e}e Poussin. This theorem provides a necessary and sufficient condition for the uniform integrability of a family of $L^1$-integrable functions (see e.g.~\cite[Theorem 1.2]{Rao-Ren}). By applying a suitable modification of this theorem, we show that for a subadditive function $\fL$ on $G$ and a probability measure $\mu$ with finite $\log(1+\fL)$-moment, one can construct another subadditive function that grows faster than $\fL$ and is still integrable with respect to $\mu$. Together with a novel result, this implies that $\Lc_{1+\fL}(G,\mu)$ vanishes. Concretely, we use the following theorem, which is a combination of the aforementioned modification of the theorem of de la Vall\'{e}e Poussin (yielding the existence of a suitable function $\Psi$ satisfying \eqref{Eq:function De LA V Poussin-2 as used} below) and a novel part (showing that the function given by $F(y) = \Psi(\log(1+y))$ satisfies the given properties), for which we give a proof in Appendix \ref{sec:delavalleepoussin}.

\begin{thm} \label{T:De LA V Poussin as used}
Let $(X,\xi)$ be a probability space, and let $f\in L^1(X,\xi)$ be a non-negative function. Then there is an increasing $1$-Lipschitz function $\psi : [0,\infty)\to [0,\infty)$ such that $\lim_{y\to \infty} \psi(y)=\infty$ and such that the following holds:
If we let $\Psi : [0,\infty) \to [0,\infty)$ be the function given by
\begin{align*}
\Psi(y)=\int_{0}^{y} \psi(z)dz \;\; \textrm{for} \;\; y\geq 0,
\end{align*}
then
\begin{align} \label{Eq:function De LA V Poussin-2 as used}
\Psi(f)\in L^1(X,\xi).
\end{align}
Moreover, the function $F : [0,\infty) \to [0,\infty)$ defined by $F(y)=\Psi(\log (1+y))$ is increasing and differentiable on $[0,\infty)$, and there exists $M > 0$ such that
\begin{align} \label{eq:F weak subadditivity}
  F(y+y') \leq F(y)+F(y')+M \;\;\textrm{for all}\;\; y,y'\geq 0.
\end{align}
\end{thm}

Now we can prove the main result of this section.

\begin{thm}\label{T:vanishing Lyap expo-log of lenght function}
Let $\fL$ be a subadditive function on $G$, and let $\mu$ be a probability measure on $G$ with finite $\log(1+\fL)$-moment. Then the Lyapunov exponent of $\mu$ with respect to $1+\fL$ vanishes. In other words,
\begin{align*}
\Lc_{1+\fL}(G,\mu)=\displaystyle\lim_{n\rightarrow \infty}\dfrac{1}{n}\sum_{s\in G}\mu^{\ast n}(s) \, \log (1+\fL(s))=0.
\end{align*}
\end{thm}

\begin{proof}
Consider the probability space $(G,\mu)$. By hypothesis, $\log(1+\fL)\in L^1(G,\mu)$. Suppose that $\Psi$ and $F$ are functions as given by Theorem \ref{T:De LA V Poussin as used} for $f=\log(1+\fL)$. By \eqref{Eq:function De LA V Poussin-2 as used}, we have
\begin{align}\label{Eq:vanishing Lyap expo lenght-1}
\Psi(\log(1+\fL))\in L^1(G,\mu).
\end{align}
By abuse of notation, we write $F(\fL)=\Psi(\log(1+\fL))$.

Secondly, since
\begin{align*} 
\lim_{y\to \infty} \frac{\Psi(y)}{y}=\infty
\end{align*}
and $F(y)=\Psi(\log(1+y))$, we have
\begin{align}\label{Eq:vanishing Lyap expo lenght-2}
\lim_{y\to \infty} \frac{F(y)}{\log(1+y)}=\infty.
\end{align}
Finally, by \eqref{eq:F weak subadditivity}, there exists $M>0$ such that
\begin{align}\label{Eq:vanishing Lyap expo lenght-3}
  F(y+y') \leq F(y)+F(y')+M \;\;\textrm{for all}\;\; y,y' \geq 0.
\end{align}
We now define the function $\Theta : G \to [M,\infty)$ by
\begin{align}\label{Eq:faster growing lenght function on G}
  \Theta(s)=F(\fL(s))+M, \;\;s\in G.
\end{align}
Since $F$ is increasing on $[0,\infty)$ and $\fL$ is subadditive on $G$, it follows from \eqref{Eq:vanishing Lyap expo lenght-3} that $\Theta$ is also subadditive on $G$. Indeed, for every $s,t\in G$, we have
\begin{align*}
  \Theta(st) &= F(\fL(st))+M \leq F(\fL(s)+\fL(t))+M \\ &\leq F(\fL(s))+F(\fL(t))+2M = \Theta(s)+\Theta(t).
\end{align*}
Moreover, by \eqref{Eq:vanishing Lyap expo lenght-1}, $\mu$ has finite first moment with respect to $\Theta$, so that $\Lc_{\zeta}(\mu,G)$, with $\zeta=e^{\Theta}$, exists.
Now, let $\varepsilon>0$ be arbitrary. By \eqref{Eq:vanishing Lyap expo lenght-2}, there exists $N>0$ such that for all $y\geq N$,
\begin{align}\label{Eq:vanishing Lyap expo lenght-4}
   \frac{\log(1+y)}{F(y)+M} <\varepsilon.
\end{align}
For every $n\in \N$, let
\begin{align*} 
  A_n:=\{(s_1,\ldots, s_n)\in G^n : \fL(s_1\cdots s_n)<N \}.
\end{align*}
Then we have
\begin{align*}
  \sum_{(s_1,\ldots,s_n)\in A_n} &\mu(s_1)\cdots \mu(s_n) \log (1+\fL(s_1\cdots s_n)) \\ &\leq  \log(1+N)  \sum_{(s_1,\ldots,s_n)\in A_n} \mu(s_1)\cdots \mu(s_n) \\ &\leq \log(1+N).
\end{align*}
Using the properties of the convolution product, we obtain
\begin{align*}
  \sum_{s\in G} &\mu^{*n}(s)\log (1+\fL(s)) =\sum_{s_1,\ldots,s_n\in G} \mu(s_1)\cdots \mu(s_n) \log (1+\fL(s_1\cdots s_n)) \\
   & \leq \log(1+N) +\sum_{(s_1,\ldots,s_n)\in A_n^c} \mu(s_1)\cdots \mu(s_n) \log (1+\fL(s_1\cdots s_n)) \\
  &\leq  \log(1+N) +\varepsilon \sum_{(s_1,\ldots,s_n)\in A_n^c} \mu(s_1)\cdots \mu(s_n) \Theta(s_1\cdots s_n) \ \ (\text{by}\ \eqref{Eq:faster growing lenght function on G}\ \text{and}\ \eqref{Eq:vanishing Lyap expo lenght-4})\\
  &\leq  \log(1+N) +\varepsilon \sum_{s\in G} \mu^{*n}(s)\Theta(s).
\end{align*}
By dividing both sides by $n$ and letting $n$ tend to $\infty$, we obtain the following:
\begin{align*}
  \Lc_{1+\fL}(G,\mu) &= \lim_{n\to \infty} \frac{1}{n} \sum_{s\in G}\mu^{*n}(s)\log (1+\fL(s)) \\
  &\leq \lim_{n\to \infty} \frac{\log(1+N)}{n}  +\varepsilon \lim_{n\to \infty} \frac{1}{n} \sum_{s\in G} \mu^{*n}(s)\Theta(s) \\
  &=\varepsilon \, \Lc_{\zeta}(G,\mu),
\end{align*}
where, as before, $\zeta=e^{\Theta}$. The final result now follows, as $\varepsilon>0$ was arbitrary.
\end{proof}

\begin{rem}
If we consider instead of the assumption in Theorem \ref{T:vanishing Lyap expo-log of lenght function} the stronger assumption that for some $\alpha >0$, the function $(1+\fL)^\alpha$ is $\mu$-integrable, then the vanishing of the Lyapunov exponent of $\mu$ with respect to $(1+\fL)$ can be deduced in another way. Indeed, it then follows from \cite[Corollary 1]{Pytlik} that for sufficiently large $p$, the spectral radius of $\mu$, which appears in \eqref{Eq:Lyap expo-sepc radius weight L1} with $\om=1+\fL$, equals $1$, so that its logarithm vanishes. Hence, in that case, the result follows directly from Theorem \ref{T:Lyap expo-weighted L1}.
\end{rem}

\section{Weighted entropy of random walks}

In this section, we consider a generalization of the concept of Avez entropy by taking into account a weight on the underlying group.

\begin{defn}\label{D:weighted Shannon entropy}
Let $\om$ be a weight on $G$, and let $\mu$ be a probability measure on $G$.
The \textbf{weighted Shannon entropy} of $\mu$ with respect to $\om$ is defined as
\begin{equation}\label{Eq:weighted Shannon entropy-defn}
H_{\omega}(G,\mu):=-\sum_{s\in G}\mu (s)\log [\mu (s)\omega (s)].
\end{equation}
\end{defn}

\begin{rem}\label{R:weighted Shannon entropy vs relative entropy}
The quantity $H_{\omega}(G,\mu)$ corresponds to the additive inverse of the Kullback--Leibler divergence (also called relative entropy) $D_{\textrm{KL}}(\mu \ \vert\vert \ \frac{1}{\om})$ of $\mu$ with respect to $1 / \om$, i.e.
\begin{equation*}
H_{\omega}(G,\mu) = -D_{\textrm{KL}}\left(\mu \ \Big\vert\Big\vert \ \frac{1}{\om}\right).
\end{equation*}
Relative entropy has been studied extensively in statistics and plays a role in numerous practical applications; see e.g.~\cite{mackay,watrous} for applications in (quantum) information theory. In case $\om^{-1}$ is also a probability measure, the Kullback--Leibler divergence can be viewed as a notion of statistical distance between two probability measures.
\end{rem}

\begin{rem} \label{R:defn weighted entropy}
It is possible that the weighted Shannon entropy of a probability measure does not exist. However, since we can rewrite \eqref{Eq:weighted Shannon entropy-defn} as
\begin{equation}\label{Eq:weighted entropy vs entropy and log moment}
H_{\omega}(G,\mu)=H(G,\mu)-\displaystyle\sum_{s\in G}\mu (s)\log \omega (s),
\end{equation}
we see that the weighted Shannon entropy exists if $\mu$ has both finite entropy and finite $\log \om$-moment. In other words, if $\mu$ has finite $\log \om$-moment, then it has finite entropy if and only if it has finite weighted Shannon entropy with respect to $\om$.
\end{rem}

It follows from the asymptotic behavior of the R\'{e}nyi entropy that the Shannon entropy satisfies the following asymptotic relation of the norm of $\mu$ in $\ell^q(G)$-spaces:
\begin{equation*} 
H(G,\mu)=\displaystyle\lim_{p\rightarrow \infty} -p\log \Vert \mu \Vert_{q},
\end{equation*}
where $q$ satisfies $\dfrac{1}{p}+\dfrac{1}{q}=1$. We now show that a similar relation holds for the weighted Shannon entropy. However, since probability measures on $G$ may not lie in weighted $\ell^q$-spaces (for unbounded weights), we must consider a moderate growth condition on $\mu$ to ensure that it eventually does. Recall from \eqref{eq:omegap} that given a weight $\om$ on $G$, the weight $\om_p$ is given by $\om_p(s)=\om(s)^{\frac{1}{p}}$.

\begin{thm} \label{T:compare weighted and non-weighted Shannon entropy}
Let $\om$ be a weight on $G$, and let $\mu$ be a probability measure on $G$ such that $\mu\in \ell^{q_0}(G,\omega_{p_0})$, with $1< q_0\leq p_0 <\infty$ and $\frac{1}{p_0}+\frac{1}{q_0}=1$.
\begin{enumerate}[(i)]
    \item The mapping
\begin{equation} \label{Eq:asymptote increasing-log weighted Lq norms}
[p_0,\infty) \rightarrow \R,\ \ p \mapsto -p \log \Vert \mu\Vert_{q,\omega_{p}},
\end{equation}
where $p$ and $q$ satisfy $\dfrac{1}{p}+\dfrac{1}{q}=1$, is an increasing function on $[p_0,\infty)$.
\item If, in addition, $\mu$ has both finite entropy and finite $\log \om$-moment, then
\begin{equation} \nonumber
H_{\omega}(G,\mu)=\displaystyle\lim_{p\rightarrow \infty} -p \log \Vert \mu \Vert_{q,\omega_{p}},
\end{equation}
where again $\dfrac{1}{p}+\dfrac{1}{q}=1$.
\end{enumerate}
\end{thm}

\begin{proof}
(i): Since $\mu\in \ell^1(G) \cap \ell^{q_0}(G,\omega_{p_0})$, by the interpolation relation \eqref{Eq:complex interpolation-weight Lq}, $\mu \in \ell^{q}(G,\omega_{p})$ for every
$1 < q\leq q_0\leq p_0 \leq p < \infty$ with $\dfrac{1}{p}+\dfrac{1}{q}=1$. Moreover, by \eqref{Eq:complex interpolation-weight Lq-norm relations}, for every
$1 < q_2\leq q_1 \leq q_0\leq p_0 \leq p_1 \leq p_2 < \infty$, where $q_1$ and $q_2$ are the conjugates of $p_1$ and $p_2$, respectively, we have
\begin{align*}
  \|\mu\|_{q_2,\om_{p_2}}\leq \|\mu\|_{q_1,\om_{p_1}}^\theta, \ \ \text{with}\ \ \theta=\frac{p_1}{p_2}.
\end{align*}
In particular,
\begin{align*}
  \|\mu\|_{q_1,\om_{p_1}}^{-p_1} \leq \|\mu\|_{q_2,\om_{p_2}}^{-p_2} \ \ \text{whenever}\ \ p_1\leq p_2,
\end{align*}
so that the mapping \eqref{Eq:asymptote increasing-log weighted Lq norms}
is an increasing function on $[p_0,\infty)$.

(ii): For $1<q\leq q_0 \leq p_0 \leq p < \infty $ with $\dfrac{1}{p}+\dfrac{1}{q}=1$, we have
\begin{align*}
-p\log \Vert \mu \Vert_{q,\omega_{p}} & = - \frac{\log \displaystyle\sum_{s\in G}\mu (s)^{q}\omega_{p}(s)^{q}}{q-1}
 =-\frac{\log \displaystyle\sum_{s\in G}\mu (s)^{q}\omega (s)^{\frac{q}{p}}}{q-1} \\
& =-\frac{\log \displaystyle\sum_{s\in G}\mu (s)^{q}\omega (s)^{q-1}}{q-1}
 =-\frac{\log \displaystyle\sum_{s\in G}[\mu (s)\omega(s)]^{q-1} \mu(s)}{q-1}.
\end{align*}
Hence, by Lemma \ref{R:derivative log-power}, with $(X,\xi)=(G,\mu)$ and $f=\mu\om$, we obtain
\begin{align*}
\lim_{p\rightarrow \infty} -p\log \Vert \mu \Vert_{q,\omega_{p}} &= \lim_{\alpha\to 0^+} -\frac{1}{\alpha} \log \displaystyle\sum_{s\in G}[\mu (s)\omega(s)]^{\alpha} \mu(s)\\
& =-\displaystyle\sum_{s\in G}\mu (s)\log [\mu (s)\omega (s)]\\
& = H_{\omega}(G,\mu).
\end{align*}
\end{proof}
Finally, we introduce the weighted Avez entropy of a probability measure. Like with weighted Shannon entropy, it is analogous to the (unweighted) Avez entropy with the Shannon entropy of convolution powers being replaced with the weighted Shannon entropy of convolution powers defined in Definition \ref{D:weighted Shannon entropy}.

\begin{defn}\label{D:weighted Avez entropy}
Let $\om$ be a weight on $G$, and let $\mu$ be a probability measure on $G$ with both finite entropy and finite $\log \om$-moment. The \textbf{weighted (Avez) entropy} of $\mu$ with respect to $\omega$ is defined as
\begin{equation}\label{Eq:weighted Avez entropy-defn}
h_{\omega }(G,\mu):= \displaystyle\lim_{n\rightarrow \infty} \dfrac{H_{\omega }(G,\mu^{\ast n})}{n}.
\end{equation}
\end{defn}

\begin{rem}

\begin{enumerate}[(i)]
\item Let $\om$ be a weight on $G$, and let $\mu$ be a probability measure on $G$ such that $\mu$ has both finite entropy and finite $\log \om$-moment. Although the mapping
$$\N\to \R, \ \ n\mapsto H_{\omega }(G,\mu^{\ast n})$$
is well-defined (by Remark \ref{R:defn weighted entropy}), it may not be subadditive. However, since
$$H_{\omega}(G,\mu^{\ast n})=H(G,\mu^{\ast n})-\displaystyle\sum_{s\in G}\mu^{\ast n} (s)\log \omega (s),$$
and since the function on the right-hand side is the difference of two subadditive functions on $\N$,
the limit in \eqref{Eq:weighted Avez entropy-defn} exists and we have
\begin{align*}
 \lim_{n\rightarrow \infty} \frac{1}{n} H_{\omega}(G,\mu^{\ast n}) &=
  \lim_{n\rightarrow \infty} \frac{1}{n}H(G,\mu^{\ast n})- \lim_{n\rightarrow \infty} \frac{1}{n} \displaystyle\sum_{s\in G}\mu^{\ast n} (s)\log \omega (s) \\
  &= h(G,\mu)-\Lc_\om(G,\mu);
\end{align*}
that is,
\begin{align}\label{Eq:weighted Avez entropy-difference entropy and Lyap expo}
h_{\omega }(G,\mu) & =h(G,\mu)-\Lc_\om(G,\mu).
\end{align}
In other words, the weighted entropy of the $\mu$-random walk with respect to $\omega$ is the difference of the (non-weighted) entropy and the Lyapunov exponent with respect to $\om$. In particular, the weighted entropy is equal to the (non-weighted) entropy if and only if the Lyapunov exponent of $\mu$ with respect to $\om$ vanishes:
\begin{align}\label{Eq:equality of weighted and nonweighted Avez entropy}
h_{\omega }(G,\mu)=h(G,\mu) \; \Longleftrightarrow \; \Lc_\om(G,\mu)=0.
\end{align}
This is a very useful relation that we will exploit in Sections \ref{S:convolution entropy} and \ref{sec:stationarydynamicalsystems}.
\item Note that $h_\om(G,\mu)$ might be a negative number. For instance, let $G$ be a finitely generated group with exponential growth with respect to a finite generating set $S$. Consider the weight $\om=a^{|\cdot|_S}$, with $\log a > v_S$, where $v_S$ is the logarithmic volume growth (see Section \ref{subsec:entropyrw}). Then for every probability measure $\mu$ on $G$ with finite first moment, we have
$$\Lc_\om(G,\mu)=(\log a)\, l(G,\mu) > h(G,\mu),$$
where $l(G,\mu)$ denotes, as before, the speed of the $\mu$-random walk, and the last inequality follows from \eqref{Eq:Fund inequality}. By \eqref{Eq:weighted Avez entropy-difference entropy and Lyap expo}, the weighted entropy of the $\mu$-random walk is negative.
\end{enumerate}
\end{rem}

There is another consequence of the relations \eqref{Eq:weighted entropy vs entropy and log moment} and \eqref{Eq:weighted Avez entropy-difference entropy and Lyap expo} that is useful and crucial for our approach. If the weighted Shannon entropy is a non-positive number, then the Shannon entropy becomes dominated by the $\mu$-integral of $\log \om$. This simple fact is very useful and we will take full advantage of it. It occurs, for instance, whenever the inverse of the weight is $\ell^1$-integrable. In fact, as we see below, we can show that the Avez entropy itself can be realized as the Lyapunov exponent of $\mu$ with respect to a particular reverse-integrable weight making it ``minimal'' among all Lyapunov exponents satisfying that property.

\begin{thm}\label{T:Lyap Expo dominates entropy-inverse integrable weights}
Let $\mu$ be a non-degenerate probability measure on $G$, and let $\Omega$ be the set of all weights $\om$ on $G$ such that $\mu$ has finite $\log \om$-moment and $\omega^{-1} = 1/\om \in \ell^1(G)$. Then the $\mu$-random walk has finite Avez entropy and
\begin{align*} 
h(G,\mu)=\inf \big\{\Lc_\om(G,\mu) : \om\in \Omega \big\}.
\end{align*}
Moreover, the infimum above is attained for a weight $\bar\om \in \Omega$ that also satisfies
\begin{align}\label{Eq:Lyap Expo-algebraic weight}
  \bar\om^{-1}*\bar\om^{-1} \leq d\, \bar\om^{-1},
\end{align}
where $d$ a positive constant. In this case, we have
\begin{align}\label{Eq:Lyap Expo equals entropy}
  h(G,\mu)=\Lc_{\bar\om}(G,\mu).
\end{align}
\end{thm}

\begin{proof}
Let $\om\in \Omega$. Recall that Gibbs' inequality states that the relative entropy of $\mu$ with respect to another probability measure $\tau$ on $G$ is positive: $D_{KL}(\mu \ || \ \tau)\geq 0$. Then, since $H_\omega(G,\mu) = -D_{KL}(\mu \ || \ \om^{-1})$ and $\om^{-1}$ is integrable on $G$, by putting $\tau=\om^{-1}/C$ with $C=\Vert \omega^{-1}\Vert_{1}$, we obtain
\begin{align}\label{Eq:claim 1}
H(G,\mu) \leq -\sum_{s\in G}\mu (s)\log \frac{\omega^{-1} (s)}{C} = \displaystyle\sum_{s\in G}\mu (s)\log \omega (s)+\log C.
\end{align}
In particular, $\mu$ has finite entropy. Also, if we replace $\mu $ with $\mu^{\ast n}$, then \eqref{Eq:claim 1} is still valid, so that we must have
$$H(G,\mu^{\ast n})\leq \sum_{s\in G}\mu^{*n}(s)\log \omega (s)+\log C , \qquad \forall ~ n\in \mathbb{N}.$$
Therefore,
$$h(G,\mu)=\displaystyle\lim_{n\rightarrow \infty}\dfrac{1}{n} H(G,\mu^{\ast n})\leq \Lc_\om(G,\mu).$$

Now we construct the weight satisfying \eqref{Eq:Lyap Expo-algebraic weight} and \eqref{Eq:Lyap Expo equals entropy}.
Define the function $\bar\om$ on $G$ by
\begin{align}\label{Eq:Reverse integrable weights}
  \bar\om^{-1}=\sum_{n=1}^{\infty} \frac{\mu^{*n}}{n^3}.
\end{align}
Then $\bar\om^{-1}\in \ell^1(G)$ and we will shortly see that $\bar\om$ is a weight on $G$. First, we observe that
\begin{align*}
  \bar\om^{-1} * \bar\om^{-1} &= \sum_{k,m=1}^{\infty} \frac{\mu^{*(k+m)}}{k^3 m^3}  \\
   & =  \sum_{n=2}^{\infty} \sum_{m=1}^{n-1} \frac{\mu^{*n}}{m^3 (n-m)^3} \\
   & =  \sum_{n=2}^{\infty}\mu^{*n} \sum_{m=1}^{n-1} \frac{1}{m^3 (n-m)^3} \\
   & \leq  \sum_{n=2}^{\infty}\frac{4\mu^{*n}}{n^3} \sum_{m=1}^{n-1} \bigg[\frac{1}{m^3}+\frac{1}{(n-m)^3}\bigg] \qquad \big(\text{as}\ n^3\leq 4(m^3+(n-m)^3)\big) \\
   &\leq  \sum_{n=2}^{\infty}\frac{4\mu^{*n}}{n^3}
   \sum_{m=1}^{\infty} \bigg[\frac{1}{m^3}+\frac{1}{m^3}\bigg]  \\
   &=8C \ \bar\om^{-1},
\end{align*}
where
$$C=\sum_{n=1}^{\infty} \frac{1}{n^3}.$$
Hence,
\begin{align*}
  \bar\om^{-1}*\bar\om^{-1} \leq 8C \ \bar\om^{-1}.
\end{align*}
This furthermore shows that
$$\bar\omega^{-1}(s)\bar\omega^{-1}(t)\leq (\bar\omega^{-1}*\bar\omega^{-1})(st)\leq 8C\bar\omega^{-1}(st)$$
for all $s,t\in G$, verifying that $\bar\omega$ is indeed a weight on $G$.
Finally, for $n \in \N$, set $B_n:=\supp \,\mu^{*n}$. Then
\begin{align*}
 - \sum_{s \in G} \mu^{*n}(s) \log \bar\om (s) & = \sum_{s \in G} \mu^{*n}(s) \log \bar\om^{-1}(s) \\
 & \geq  \sum_{s\in B_n} \mu^{*n}(s) \log \bigg(\frac{\mu^{*n}(s)}{n^3}\bigg) \\
 & =  \sum_{s\in B_n} \mu^{*n}(s) \log \mu^{*n}(s) - \sum_{s\in B_n} \mu^{*n}(s) \log n^3 \\
 & = -H(G,\mu^{*n}) - 3\log n.
\end{align*}
Therefore,
\begin{align*}
    \Lc_{\bar\om}(G,\mu) & = \displaystyle\lim_{n\rightarrow \infty}\dfrac{1}{n} \sum_{s \in G} \mu^{*n}(s) \log \bar\om(s) \\
    & \leq \displaystyle\lim_{n\rightarrow \infty}\dfrac{1}{n}  \bigg [H(G,\mu^{*n}) + 3\log n\bigg] \\
    &= h(G,\mu).
\end{align*}
This completes the proof.
\end{proof}

\begin{rem}\label{R:infinite sum-conv power of prob measure}
It is clear that the weight $\om$ constructed in \eqref{Eq:Reverse integrable weights} is not unique, as there are many possibilities for the coefficients. The idea of constructing a weight as such (and the construction in the proof of Theorem \ref{T:Lyap Expo dominates entropy-inverse integrable weights}) is given in \cite[Theorem 1.1]{Kuznets 1} in a more general setting. Interestingly, summations of convolution powers of a probability measure are repeatedly used in the theory of random walks for different reasons. Here is one situation that one considers: For a non-degenerate probability measure $\mu$ with finite entropy, it is shown in \cite[Section 3.2]{Kaim-Vers 1} that the entropy of the probability measure
$$\tilde{\mu}=\sum_{n=0}^{\infty} a_n\mu^{*n}, \;\; \text{with}\ \sum_{n=0}^{\infty} a_n=1, \;\; \sum_{n=1}^{\infty} n a_n<\infty \;\; \textrm{and} \;\; a_0\neq 0,$$ is given by
$$h(G,\tilde{\mu})=h(G,\mu) \; \sum_{n=1}^{\infty} na_n.$$
Moreover, if $(X,\xi)$ is a $(G,\mu)$-space (see Section \ref{sec:stationarydynamicalsystems}), then $(X,\xi)$ is also $\tilde{\mu}$-stationary (i.e.~a $(G,\tilde{\mu})$-space) if the numbers $a_n$ are chosen so that $\sum_{n=1}^{\infty} n a_n=1$. This, for instance, allows one to assume that $\mu$ is fully supported on the group by replacing $\mu$ with $\tilde{\mu}$.
\end{rem}

\section{Poisson boundaries and stationary spaces} \label{sec:stationarydynamicalsystems}

\subsection{$L^p$-Koopman representation}
Let $(X,\xi)$ be a standard probability space equipped with a measurable $G$-action. We assume the action to be non-singular (also called quasi-invariant), i.e.~for every $s\in G$, the measures $s\xi$ and $\xi$ are mutually absolutely continuous (where $(s\xi)(A)=\xi(s^{-1}A)$ for measurable subsets $A \subset X$). For $1 \leq p < \infty$, the {\textbf{$L^p$-Koopman representation}} of $G$ on $L^p(X,\xi)$ is the map
\begin{align*} 
    \pi_{p,X}: G\to B(L^p(X,\xi))
\end{align*}
given by
\begin{align*} 
[\pi_{p,X}(s)f](x)=\left[\frac{d(s\xi)}{d\xi}(x)\right]^{1/p} f(s^{-1}x)
\end{align*}
for all $s \in G$, $f\in L^p(X,\xi)$ and $\xi$-a.e.~$x$. Here,
$\frac{d(s\xi)}{d\xi}$ denotes the Radon-Nikodym derivative of $s\xi$ with respect to $\xi$.
For simplicity, we write
\begin{align*} 
    \rho(s,x)=\frac{d(s^{-1}\xi)}{d\xi}(x), \;\; s\in G, \; x\in X.
\end{align*}
It is well known that $\rho$ satisfies the following cocycle relation:
\begin{align*} 
    \rho(st,x)=\rho(s,tx)\rho(t,x) \;\; \textrm{for all} \;\; s,t\in G \textrm{ and } x\in X.
\end{align*}
By applying this relation, it is straightforward to verify that if $1 < p < \infty$ and $q$ is the conjugate exponent of $p$, then
\begin{align} \label{Eq:Lq Lp-Koopman rep-formula}
    \la \pi_{p,X}(s)f,g \ra =\la f,\pi_{q,X}(s^{-1})g \ra=\int_X \rho(s,x)^{1/q} f(x) \overline{g(sx)}d\xi(x)
\end{align}
for all $f\in L^p(X,\xi)$ and $g\in L^q(X,\xi)$. For simplicity, we write $\pi_X:=\pi_{2,X}$.

Following \cite{Nevo1}, we consider the function $\Xi_{X} : G \to [0,\infty)$ associated with $G \curvearrowright (X,\xi)$ defined as the diagonal matrix coefficient of $\pi_{X}$ with respect to the unit vector $1\in L^\infty(X,\xi)\subseteq L^2(X,\xi)$, i.e.
\begin{align} \label{Eq:Harich-Chandra function}
  \Xi_{X}(s):=\la \pi_{X}(s) 1, 1 \ra=\int_{X} \rho(s,x)^{1/2}  d\xi(x), \;\; s\in G.
\end{align}
The function $\Xi_{X}$ is a generalization of the \textbf{Harish-Chandra $\Xi$-function}. It is straightforward to see that $\Xi_X$ only attains non-negative values and that $\Xi_{X}(s^{-1})=\Xi_{X}(s)$ for all $s\in G$.

\subsection{Stationary dynamical systems}

As above, let $(X,\xi)$ be a standard probability space equipped with a non-singular $G$-action. Let $\mu$ be a non-degenerate probability measure on $G$. We say that $(X,\xi)$ is a \textbf{stationary} $(G,\mu)$\textbf{-space} (or that $(G,\mu) \curvearrowright (X,\xi)$ is a \textbf{stationary dynamical system}) if $\xi$ is a $\mu$-stationary measure, i.e.~$\mu \ast \xi = \xi$, where $\mu \ast \xi (A) = \sum_{s \in G} \mu(s) \xi(s^{-1}A)$ for measurable subsets $A \subset X$. (In fact, the assumption that the action is non-singular above is not necessary, as it holds automatically for stationary dynamical systems.) The notion of stationary dynamical system goes back to \cite{FurstenbergNCRP}, and we refer to \cite{Furs-Glas 1} for an overview.

The \textbf{Furstenberg entropy} of $(X,\xi)$ with respect to $(G,\mu)$, going back to \cite{FurstenbergNCRP} (see also \cite{NevoZimmer,Kaim-Vers 1}), is defined as
\begin{align*} 
  h_\mu(X,\xi)=- \sum_{s\in G} \mu(s) \int_X  \log \rho(s,x) d\xi(x).
\end{align*}
It is well known that $h_\mu(X,\xi)\geq 0$, with equality if and only if $\xi$ is $G$-invariant. Also, it was shown by Kaimanovich and Vershik \cite{Kaim-Vers 1} that if $\mu$ has finite Shannon entropy, then
\begin{align} \label{Eq:maximality of Furstenberg boundary}
h(G,\mu)=h_\mu(\Pi_{\mu},\nu_{\infty})\geq h_\mu(X,\xi),
\end{align}
where $(\Pi_{\mu},\nu_{\infty})$ is the Poisson boundary of $(G,\mu)$ and $(X,\xi)$ is an arbitrary $(G,\mu)$-stationary space. Moreover, for every $\mu$-boundary $(X,\xi)$ (see e.g.~\cite{Furs-Glas 1} for the terminology), the equality in \eqref{Eq:maximality of Furstenberg boundary} occurs if and only if there is a measurable $G$-equivariant isomorphism from $(X,\xi)$ onto $(\Pi_{\mu},\nu_{\infty})$.

The Furstenberg entropy is a powerful invariant. In the following proposition, we show that for non-degenerate probability measures, it can be computed as an asymptotic limit of matrix coefficients of the form $\la \pi_{p,X}(\mu) 1, 1 \ra$ (where $\pi_{p,X}(\mu)f = \sum_{s \in G} \mu(s) \pi_{p,X}(s)f$). The formula is essential for our approach and it resembles the one for $h(G,\mu)$ established in \eqref{Eq:Avez-entropy-p conv operators-RD}.

\begin{prop} \label{P:asymptote Lp Harish-Chandra function to entropy}
Let $(X,\xi)$ be a $(G,\mu)$-stationary space. Then the mapping
\begin{align} \label{Eq:asymptotic Lp Harish-Chandra function}
[2, \infty )\rightarrow [0,\infty ), \;\; p \mapsto -p\log \la \pi_{q,X}(\mu) 1, 1 \ra,
\end{align}
where $q$ is the conjugate exponent of $p$, is an increasing function with
\begin{align}\label{Eq:asymptotic Lp Harish-Chandra function to entropy}
\lim_{p\to \infty} -p\log \la \pi_{q,X}(\mu) 1, 1 \ra=h_\mu(X,\xi).
\end{align}
\end{prop}

\begin{proof}
  Let $2\leq p_1\leq p_2<\infty$, and let $q_1$ and $q_2$ be their conjugate exponents, respectively.
Then, by \eqref{Eq:Lq Lp-Koopman rep-formula} and Jensen's inequality, we have
  \begin{align*}
     \la \pi_{q_2,X}(s) 1, 1 \ra^{\frac{p_2}{p_1}} & =
     \bigg(\int_{X} \rho(s,x)^{1/p_2} \, d\xi(x)\bigg)^{\frac{p_2}{p_1}} \\
     & \leq\int_{X} \big(\rho(s,x)^{1/p_2}\big)^{\frac{p_2}{p_1}}\,
     d\xi(x) \\ &= \la \pi_{q_1,X}(s) 1, 1 \ra.
  \end{align*}
  Hence,
  $$ \la \pi_{q_2,X}(s) 1, 1 \ra^{p_2} \leq  \la \pi_{q_1,X}(s) 1, 1 \ra^{p_1} \;\; \textrm{for all} \;\; s\in G.$$
By taking minus the logarithm and summing with respect to $\mu$, we see that the mapping in \eqref{Eq:asymptotic Lp Harish-Chandra function} is increasing on $[2,\infty)$.
Moreover, by applying Lemma \ref{R:derivative log-power} to the probability space $(G\times X, \mu\times \xi)$ and the function $\rho$, we obtain
\begin{align*}
\lim_{p\to \infty} -p\log \la \pi_{q,X}(\mu) 1, 1 \ra &=  \lim_{\alpha\to 0^+} -\frac{1}{\alpha} \log \displaystyle\int_{G\times X} \rho(s,x)^\alpha \, d(\mu\times \xi)(s,x)\\
& = -\int_{G\times X} \log \rho(s,x) \, d(\mu\times \xi)(s,x).
\end{align*}
\end{proof}

We are now ready to present one of the main results of this article. We show that for a group with rapid decay,
the sufficiently fast decay of the function $\Xi_{X}$ (with respect to the length function) associated with $(X,\xi)$ ensures that the Furstenberg entropy of $(X,\xi)$ must be equal to the Avez entropy of the $\mu$-random walk.

\begin{thm} \label{T:Square integrable-inverse Harich Chandra function}
Let $G$ be a countable group satisfying property RD with respect to a length function $\mathcal{L}$, and let $\mu$ be a non-degenerate probability measure on $G$ with both finite entropy and finite $\log(1+\fL)$-moment. Let $(X,\xi)$ be a $(G,\mu)$-stationary space such that $\pi_X$ is weakly contained in the left-regular representation $\lambda_G$. Then
\begin{align*}
h_\mu(X,\xi)=h(G,\mu)=-2\lim_{n\rightarrow \infty}\dfrac{1}{n}\sum_{s\in G}\mu^{\ast n}(s) \log \Xi_X(s).
\end{align*}
\end{thm}

\begin{proof}
Since $G$ has property RD with respect to $\fL$, there exists $d>0$ such that for the weight $\om$ on $G$ given by
$\omega(s)=(1+\fL(s))^d$ for $s\in G$, we have $\ell^2(G,\om)\subseteq C^*_r(G)$.
This, together with the fact that $\pi_X$ is weakly contained in $\lambda_G$, implies that the matrix coefficients of $\pi_X$ lie inside $\ell^2(G,\om^{-1})$. (This is related to a characterization of property RD in terms of the matrix coefficients of $\lambda_G$, which is stated as $RD(6)$ in \cite{Chatt 1} and attributed to Breuillard.) In particular, by \eqref{Eq:Harich-Chandra function}, we have that $\Xi_X\in \ell^2(G,\om^{-1})$, or equivalently,
\begin{align} \label{Eq:square invertible-Harich Chandra function}
 \frac{\Xi_{X}}{(1+\fL)^{d}}\in \ell^{2}(G).
\end{align}
Let
\begin{align} 
  \sg:=\frac{(1+\fL)^{d}}{\Xi_{X}}.
\end{align}
Although we do not know whether $\sg$ is a weight on $G$, a careful examination of the proof of the relation \eqref{Eq:claim 1} shows that the integrability of $\sg^{-2}$ (which follows from \eqref{Eq:square invertible-Harich Chandra function}) suffices to imply that
$$H(G,\mu^{\ast n})\leq \sum_{s\in G}\mu^{*n}(s)\log \sg^2 (s)+D \qquad \forall n\in \mathbb{N},$$
where $D=\log\|\sg^{-2}\|_1$. Therefore,
\begin{align*}
  h(G,\mu) &= \inf_{n\in \N} \left\{\frac{1}{n} H(G,\mu^{\ast n}) \right\} \\
  & \leq 2 \inf_{n\in \N} \left\{\dfrac{1}{n}\sum_{s\in G}\mu^{*n}(s)\log \sg (s)\right\}  \\
  & =2 \inf_{n\in \N} \left\{\dfrac{d}{n}\sum_{s\in G}\mu^{*n}(s)\log (1+\fL(s))-  \dfrac{1}{n}\sum_{s\in G}\mu^{\ast n}(s) \log \Xi_X(s) \right\} \\
  & \leq 2 \liminf_{n\to \infty} \left[\dfrac{d}{n}\sum_{s\in G}\mu^{*n}(s)\log (1+\fL(s))-  \dfrac{1}{n}\sum_{s\in G}\mu^{\ast n}(s) \log \Xi_X(s) \right].
  \end{align*}
  However, by Theorem \ref{T:vanishing Lyap expo-log of lenght function},
$$\Lc_{1+\fL}(G,\mu)=0.$$
Hence,
\begin{align*}
 h(G,\mu) & \leq  2 \liminf_{n\to \infty} \left[\dfrac{1}{n}\sum_{s\in G}\mu^{\ast n}(s) (-\log) \Xi_X(s) \right] \\
  & \leq 2 \limsup_{n\to \infty} \left[\dfrac{1}{n}\sum_{s\in G}\mu^{\ast n}(s) (-\log) \Xi_X(s) \right] \\
   &\leq 2\limsup_{n\rightarrow \infty}\dfrac{1}{n}\sum_{s\in G}\mu^{\ast n}(s) \int_X (-\log) \big[\rho(s,x)^{1/2}\big] d\xi(x) \ \ (\text{as} -\log\ \text{is convex}) \\
    &=2\lim_{n\rightarrow \infty}\dfrac{1}{n}\sum_{s\in G}\mu^{\ast n}(s) \int_X \bigg(-\frac{1}{2}\bigg)
    \log \rho(s,x) d\xi(x) \\
  &=2 \lim_{n\rightarrow \infty} \frac{1}{2n} h_{\mu^{*n}}(X,\xi)\\
    &= h_\mu(X,\xi) \\
    &\leq h(G,\mu).
  \end{align*}
Therefore,
 $$h_\mu(X,\xi)=h(G,\mu)=-2\lim_{n\rightarrow \infty}\dfrac{1}{n}\sum_{s\in G}\mu^{\ast n}(s) \log \Xi_{X}(s).$$
\end{proof}

We now apply the preceding theorem to characterize amenable actions of groups with property RD for the class of measures considered in Theorem \ref{T:Square integrable-inverse Harich Chandra function}.
Let us first recall the following definitions.

\begin{defn} \label{D:factors}
Let $(X,\xi)$ and $(Y,\eta)$ be measurable spaces, and let $\phi \colon (X,\xi)\to (Y,\eta)$ be a measurable function. The push-forward measure $\phi_*(\xi)$ of $\xi$ under $\phi$ is the measure on $Y$ defined by
\begin{align*}
  \phi_*(\xi)(A)= \xi(\phi^{-1}(A)), \qquad A \subset Y \ \text{measurable}.
\end{align*}
If, in addition, $(X,\xi)$ and $(Y,\eta)$ are $G$-spaces and $\phi$ is $G$-equivariant with $\eta=\phi_*(\xi)$, then $\phi$ is called a \textbf{factor}, and $(X,\xi)$ is called an \textbf{extension} of $(Y,\eta)$. Suppose that $(X,\xi)$ and $(Y,\eta)$ are measurable $G$-spaces and $\phi \colon (X,\xi)\to (Y,\eta)$ is a factor. Then $\phi$ implies the \textbf{unique disintegration} of $\xi$ with respect to $\eta$. More precisely, there exists a unique measurable map $D:Y\to P(X)$ (where $P(X)$ denotes the space of probability measures on $X$) such that for $\eta$-a.e.~$y\in Y$, the measure $\xi_y:=D(y)$ is a probability measure supported on the fiber $\phi^{-1}(y)$ in $X$ and has the property that the barycenter of the measure
$D_*(\eta)$ on $P(X)$ is $\xi$. In this case, we write
$\xi=\displaystyle\int_{Y} \xi_y \, d\eta(y)$ and this decomposition is unique
up to $\eta$-null sets. We say that the map $\phi$ is \textbf{measure-preserving} if $D$ is $G$-equivariant, or equivalently, for every $s \in G$, we have $s\xi_y=\xi_{sy}$ for $\eta$-a.e.~$y$.
We refer the reader to \cite[Section 2]{badershalom} and \cite[Definition 1.7]{NevoZimmer} for further details.
\end{defn}

It is well known that the factor $\phi:(X,\xi)\to (Y,\eta)$ is measure preserving if and only if $h_\mu(X,\xi)=h_\mu(Y,\eta)$; see \cite[Theorem 1.3.(2)]{Furs-Glas 1}.

\begin{thm}\label{T:charc amen action-RD groups}
Let $G$ be a countable group satisfying property RD with respect to a length function $\fL$, let $\mu$ be a non-degenerate probability measure on $G$ with both finite entropy and finite $\log(1+\fL)$-moment, and let $(X,\xi)$ be a stationary $(G, \mu)$-space. Then the following are equivalent:
\begin{enumerate}[(i)]
    \item The space $(X,\xi)$ is a measure-preserving extension of the Poisson boundary of $(G,\mu)$.
    \item The action $G \curvearrowright (X,\xi)$ is amenable (in the sense of Zimmer).
    \item The representation $\pi_{X}$ is weakly contained in $\lambda_G$.
    \item $h_\mu(X,\xi)=h(G,\mu)$.
\end{enumerate}
\end{thm}

\begin{proof}
(i) $\Rightarrow$ (ii): It is well known that groups acts amenably on their Poisson boundary \cite{Zimm 2}. Hence, the implication follows from the fact that measure-preserving extensions of amenable actions are amenable \cite[Corollary C]{Adam-Elliot-Gior} (see \cite[Theorem 2.4]{Zimm 2} for the special case of ergodic actions and \cite[Theorem 9.2]{nevosageev} for a characterization of the Poisson boundary as the unique minimal amenable $(G,\mu)$-stationary space).

(ii) $\Rightarrow$ (iii): This is a general consequence of the action being amenable \cite[Theorem 4.3.1]{Ana-Dela 1}.

(iii) $\Rightarrow$ (iv): This follows directly from Theorem \ref{T:Square integrable-inverse Harich Chandra function}.

(iv) $\Rightarrow$ (i): This follows from \cite[Theorem 4.4]{Furs-Glas 1}.
\end{proof}

The above theorem does not hold in general for groups without property RD. For instance, for amenable groups $G$, assertions (ii) and (iii) always holds, whereas assertion (iv) is not automatic.

\begin{rem}
We now present concrete examples of groups (without property RD) for which Theorem \ref{T:charc amen action-RD groups} fails.
Let $G$ be a finitely generated solvable group of exponential growth (e.g.~the Lamplighter group), and let $S$ be a finite symmetric generating subset of $G$. The group $G$ does not satisfy property RD, as amenable groups with property RD must have polynomial growth \cite{Joli}. Let $\mu=\frac{1}{|S|}1_S$, i.e.~the uniform probability measure on $S$. It is well known that $S$ must contain two elements that generate a free subsemigroup. In particular, this implies that the Avez entropy of $\mu$ is nonzero. However, the trivial action of $G$ on any singleton set $\{ \ast \}$ is $G$-invariant, so that its Furstenberg entropy is $0$; that is, assertion (iv) of Theorem \ref{T:charc amen action-RD groups} fails. On the other hand, as $G$ is amenable, its action on $\{ \ast \}$ is also amenable, so that assertions (ii) and (iii) hold.
\end{rem}

The following is an immediate corollary of the preceding theorem, which can be viewed as a ``representation theoretic version'' of the approach taken in \cite{Kaim 1}.

\begin{cor} \label{C:charc entopry-RD groups}
Suppose that $G$ has property RD with respect to $\fL$ and that $\mu$ is a non-degenerate probability measure on $G$ with both finite entropy and finite $\log (1+\fL)$-moment. Let $(X,\xi)$ be a $\mu$-boundary. Then the following are equivalent:
\begin{enumerate}[(i)]
    \item The action $G \curvearrowright (X,\xi)$ is amenable (in the sense of Zimmer).
    \item The representation $\pi_{X}$ is weakly contained in $\lambda_G$.
    \item The space $(X,\xi)$ is the Poisson boundary of $(G,\mu)$.
\end{enumerate}
\end{cor}

\begin{proof}
This corollary follows directly from Theorem \ref{T:charc amen action-RD groups} together with the fact that for a $\mu$-boundary $(X,\xi)$, we have
$h_\mu(X,\xi)=h(G,\mu)$ if and only if $(X,\xi)$ is the Poisson boundary of $(G,\mu)$
(see \cite[Theorem 1.3.(3)]{Furs-Glas 1}).
\end{proof}

\section{Convolution entropy} \label{S:convolution entropy}

\subsection{Avez entropy as a convolution-type entropy }\label{S:Avez entopy vs convolution-type entropy} 
In this section, we show how the entropy of a $\mu$-random walk on a group $G$ relates to the asymptotic behavior of the spectral radius of the random walk in the Banach algebra of $p$-pseudofunctions. Moreover, for groups with property RD, we prove that the entropy of the random walk can be computed in terms of this asymptotic behavior.

For $1 < p < \infty$, let $\lambda_p : \ell^1(G)\to B(\ell^p(G))$ be the left-regular representation, i.e.~the representation defined by
$$ \lambda_p(f)g=f*g.$$
The representation $\lambda_p$ is a contractive representation and the norm closure of $\lambda_p(\ell^1(G))$ inside $B(\ell^p(G))$ is called the algebra $\pf_p(G)$ of \textbf{$p$-pseudofunctions}, which goes back to \cite{Herz}. For simplicity, we write $\lambda=\lambda_2$, which is the left-regular representation of $\ell^1(G)$ on $\ell^2(G)$. By interpolation (see Section \ref{subsec:weightedgroupalgebras}), we have
\begin{align}\label{Eq:complex interpolation-convlution operators-norm relations}
  \|f\|_{\pf_u(G)}\leq \|f\|_1^{1-\theta}\|f\|_{\pf_p(G)}^\theta, \,\, \text{with} \,\, \theta=\frac{u}{p},
\end{align}
for $1\leq u\leq p$ and $f\in \ell^1(G)$.

The following result states that a certain asymptotic behavior of a probability measure $\mu$ in $\pf_p(G)$ is dominated by the entropy of the $\mu$-random walk.

\begin{thm}\label{T:Avez entropy-p conv operators-I}
Let $\mu $ be a probability measure on $G$ with finite entropy. Then the mapping
\begin{align}\label{Eq:asymptote Lp conv operators}
[2,\infty )\rightarrow [0,\infty ), \;\; p \mapsto -p\log r_{\pf_q(G)}(\mu),
\end{align}
where $\frac{1}{p}+\frac{1}{q}=1$, is increasing and bounded on $[2,\infty)$, and
\begin{align*} 
 \displaystyle\lim_{p\rightarrow \infty}-p\log r_{\pf_q(G)}(\mu) \leq h(G,\mu).
\end{align*}
\end{thm}

We call the limit $\lim_{p\rightarrow \infty}-p\log r_{\pf_q(G)}(\mu)$ the \textbf{convolution entropy} of the $\mu$-random walk, denoted by $c(G,\mu)$ (see Definition \ref{d:convolutionentropy}).

\begin{proof}
For every $1<q<p<\infty $ with $\dfrac{1}{p}+\dfrac{1}{q}=1$ and $n\in \mathbb{N}$, we have
$$\Vert \mu^{\ast n}\Vert_{q}^{-p} = \Vert \mu^{\ast n} \ast \delta_e \Vert_{q}^{-p} \geq \Vert \mu^{\ast n}\Vert_{\pf_q(G)}^{-p}\geq \Vert \mu \Vert_{\pf_q(G)}^{-np} .$$
Hence,
$$-p\log \Vert \mu^{\ast n}\Vert_{q} \geq -np \log \Vert \mu \Vert_{\pf_q(G)}\geq 0. $$
By letting $q\rightarrow 1^{+}$ (and hence, $p\rightarrow \infty$), we obtain
$$H(G,\mu^{\ast n})\geq n \ A(G,\mu),$$
where
$$A(G,\mu) := \displaystyle\lim_{p\rightarrow \infty}-p\log \Vert \mu \Vert_{\pf_q(G)}.$$
We note that $A(G,\mu)$ exists, since by \eqref{Eq:complex interpolation-convlution operators-norm relations} the mapping
$$[2,\infty ) \mapsto [0,\infty ), \;\; p\mapsto -p \log \Vert \mu \Vert_{\pf_q(G)}$$
is increasing (and bounded by $H(G,\mu)$). Therefore,
$$h(G,\mu)=\displaystyle\lim_{n\rightarrow \infty }\dfrac{H(G,\mu^{\ast n})}{n} \geq A(G,\mu).$$
By replacing $\mu$ with $\mu^{\ast n}$, we obtain
\begin{align*}
h (G,\mu) =\dfrac{h(G,\mu^{\ast n})}{n} \geq \dfrac{A(G,\mu^{\ast n})}{n} \geq \dfrac{-p\log \Vert \mu^{\ast n}\Vert_{\pf_q(G)}}{n}.
\end{align*}
Combining the preceding relation with \eqref{Eq:complex interpolation-convlution operators-norm relations}, it follows that the mapping \eqref{Eq:asymptote Lp conv operators} is increasing, with
\begin{align*}
& h (G,\mu) \geq -p\log r_{\pf_q(G)}(\mu ) \;\; \textrm{for} \;\; p\geq 2.
\end{align*}
Therefore,
\begin{align} \label{eq:avezentropyequalconvolutionentropy}
h (G,\mu) \geq \displaystyle\lim_{p\rightarrow \infty}-p\log r_{\pf_q(G)}(\mu ).
\end{align}
\end{proof}

The following theorem proves that for random walks on groups with property RD, we have equality in \eqref{eq:avezentropyequalconvolutionentropy}.

\begin{thm}\label{T:Avez entropy-p conv operators-RD}
Let $G$ be a countable group satisfying property RD with respect to a length function $\fL$, and let $\mu$ be a probability measure on $G$ with both finite entropy and finite $\alpha$-moment with respect to $1+\fL$ for some $\alpha>0$. Then the Avez entropy and the convolution entropy of the $\mu$-random walk coincide:
\begin{align} \label{Eq:Avez-entropy-p conv operators-RD}
    h(G,\mu)=c(G,\mu).
\end{align}
\end{thm}

\begin{proof}
By Theorem \ref{T:Avez entropy-p conv operators-I}, we only need to prove that $h(G,\mu) \leq c(G,\mu)$.
Since $G$ has property RD with respect to $\fL$, there exists $d \in \N$ such that for the weight $\om$ on $G$ given by
$$\om(s)=(1+\fL(s))^d, \;\; s\in G,$$
we have the inclusion $\ell^2(G,\om_2)\subseteq C^*_r(G)$.
By hypothesis, $\mu\in \ell^1(G,\om^{\frac{\alpha}{d}})$, so that by (the proof of) Theorem \ref{T:Lyap expo-weighted L1}, $\mu$ has finite $\log \om$-moment. On the other hand, $\log \om=d\log(1+\fL)$, and hence, by Theorem \ref{T:vanishing Lyap expo-log of lenght function}, the Lyapunov exponent of $\mu$ with respect to $\om$ vanishes:
$$\Lc_\om(G,\mu)=0.$$
In other words, by \eqref{Eq:equality of weighted and nonweighted Avez entropy}, we have
$$h(G,\mu)=h_\omega(G,\mu).$$
Therefore, it suffices to prove that
\begin{align*}
 h_\omega(G,\mu)\leq \displaystyle\lim_{p\rightarrow \infty}-p\log r_{\pf_q(G)}(\mu).
\end{align*}
Now pick $p_0\in [2,\infty]$ such that $p_0> \displaystyle\frac{d}{\alpha}$. By our hypothesis,
$\mu \in \ell^{1}(G,\omega_{p_0})$, so that 
$$\mu \in \displaystyle\bigcap_{p\geq p_0}\ell^{1}(G,\omega_{p})\subseteq \bigcap_{p\geq p_0}\ell^{q}(G,\omega_{p}).$$
Fix $1< q\leq p<\infty$ with $p\geq p_0$ and $\dfrac{1}{p}+\dfrac{1}{q}=1$.
Using interpolation methods, it is showed in the proof of \cite[Theorem 5.10]{SW-exotic} that
$$\ell^{q}(G,\omega_{p}) \subseteq \pf_q(G).$$
Moreover, there exists a constant $C>0$ (in fact, $C$ is the norm of the inclusion map $\ell^{q_0}(G,\omega_{p_0}) \hookrightarrow \pf_{q_0}(G)$) such that for every $n\in \mathbb{N}$,
$$r_{\pf_q(G)}(\mu^{*n})\leq \Vert \mu^{\ast n}\Vert_{\pf_q(G)}\leq C^{\frac{p_0}{p}}\Vert \mu^{\ast n}\Vert_{\ell^{q}(G,\omega_{p})},$$
so that
$$-p_0\log C -p\log \Vert \mu^{\ast n}\Vert_{\ell^{q}(G,\omega_{p})} \leq -p n \log r_{\pf_q(G)}(\mu). $$
By letting $p\to \infty$, we have (by Theorem \ref{T:compare weighted and non-weighted Shannon entropy})
$$-p_0\log C - H_\omega(\mu^{\ast n},G) \leq n\lim_{p\rightarrow \infty} -p \log r_{\pf_q(G)}(\mu). $$
Finally, by dividing both sides by $n$ and letting $n\to \infty$, we obtain
\begin{align*}
 h_\omega(G,\mu)=\displaystyle\lim_{n\rightarrow \infty}\dfrac{H_{\omega}(G,\mu^{\ast n})}{n}\leq \displaystyle\lim_{p\rightarrow \infty}-p\log r_{\pf_q(G)}(\mu).
\end{align*}
This completes the proof.
\end{proof}

\begin{rem}
\
\begin{enumerate}[(i)]
\item To the knowledge of the authors, equality \eqref{Eq:Avez-entropy-p conv operators-RD} is new, even for free groups. It generalizes known results on the entropy of random walks on groups with polynomial growth to non-amenable groups with property RD. It also provides a new interesting way to compute the entropy of random walks on groups with property RD.

\item As a consequence of Theorem \ref{T:Avez entropy-p conv operators-RD}, we observe that for suitable probability measures on groups with property RD, we can ``commute the double limits'' in computing the Avez entropy. Indeed, for a probability measure $\mu$ satisfying \eqref{Eq:Avez-entropy-p conv operators-RD}, we have
\begin{align*}
 h(G,\mu) &= \lim_{n\rightarrow \infty} \lim_{q\rightarrow 1^{+}} -\frac{p}{n}\log \Vert \mu^{*n} \Vert_{q} \\
 &= \lim_{q\rightarrow 1^+} \lim_{n\rightarrow \infty} -\frac{p}{n}\log \Vert \mu^{*n} \Vert_{q}.
\end{align*}
\end{enumerate}
\end{rem}

\begin{cor}\label{C:Trivial Poisson boundary-RD}
Let $G$ be a group with property RD with respect to $\fL$, and let $\mu$ be a probability measure on $G$ with both finite entropy and finite $\alpha$-moment with respect to $\fL$ for some $\alpha>0$. Then the $\mu$-random walk on $G$ has trivial Poisson boundary if and only if $\supp\,\mu$ generates an amenable subgroup of $G$.
\end{cor}
\begin{proof}
    The ``only if'' direction is generally true for random walks on groups. For the ``if'' direction, suppose that $\supp\, \mu$ generates an amenable subgroup. Then, by Kesten's criterion, $r_{C^*_r(G)}(\mu)=1$. Hence, for every $1< q< p<\infty$ with $\dfrac{1}{p}+\dfrac{1}{q}=1$, we have (for some $\theta\in (0,1)$ depending on $q$)
    $$1=r_{\pf_2(G)}(\mu)\leq r_{\pf_q(G)}(\mu)^{1-\theta} r_{\pf_p(G)}(\mu)^{\theta}\leq r_{\pf_q(G)}(\mu)^{1-\theta} \leq 1 . $$
    Therefore, $r_{\pf_q(G)}(\mu)=1$. The final result now follows from Theorem \ref{T:Avez entropy-p conv operators-RD}.
\end{proof}

\begin{rem}
    Corollary \ref{C:Trivial Poisson boundary-RD} was already known for all finitely generated groups. Indeed, it follows from \cite{FHTV} that it is known if $\supp\, \mu$ lies inside a group without any quotient groups with the infinite conjugacy class (ICC) property. By \cite[Theorem 1]{FTV}, this property is characterized by being virtually nilpotent (or equivalently, having polynomial growth). However, for finitely generated groups, having a non-virtually nilpotent amenable subgroup is an obstruction to property RD \cite{Joli}.
\end{rem}

\subsection{Furstenberg entropy as a convolution-type entropy }

In Section \ref{S:Avez entopy vs convolution-type entropy}, we established a relation between the entropy of a $\mu$-random walk and the asymptotic behavior of its spectral radius in the algebra of convolution operators on $\ell^q(G)$-spaces. We will now show that a similar result holds for actions of groups with rapid decay on $(G,\mu)$-stationary spaces, provided that the associated Koopman representation is weakly contained in the left-regular representation of $G$.

We start with the following result, which is analogous to Theorem \ref{T:Avez entropy-p conv operators-I} and Proposition \ref{P:asymptote Lp Harish-Chandra function to entropy}.

\begin{prop}\label{P:asymptote Lp conv operators-(X,xi)}
Let $(X,\xi)$ be a $(G,\mu)$-stationary space. Then the mapping
\begin{align}\label{Eq:asymptote Lp conv operators-(X,xi)}
[2,\infty) \rightarrow [0,\infty ), \;\; p \mapsto -p\log \|\pi_{q,X}(\mu)\|,
\end{align}
where $q$ is the conjugate exponent of $p$, is an increasing function bounded by $h_\mu(X,\xi)$.
\end{prop}

\begin{proof}
Let $2 \leq p_1 \leq p_2 <\infty$, and let $q_1$ and $q_2$ be their respective conjugate exponents. For every $f\in L^{q_2}(X,\xi)$ and $g\in L^{p_2}(X,\xi)$, we have
\begin{align*} 
 \bigg| \sum_{s\in G}\int_X \mu(s)f(x)\overline{g(sx)}\rho(s,x)^{1/p_2} d\xi(x)\bigg| \leq  \|\pi_{q_2,X}(\mu)\| \|f\|_{L^{q_2}(X,\xi)}\|g\|_{L^{p_2}(X,\xi)}.
\end{align*}
In other words, the mapping
\begin{align}\label{Eq:bilinear mapping-conv operator}
  (f,g)\mapsto T_\mu(f,g)(s,x)=\mu(s)f(x)g(sx), \;\; s\in G, x\in X,
\end{align}
defines a bounded bilinear map
\begin{align}\label{Eq:bilinear mapping-Con operator-Lq2}
  T_{q_2,p_2}:  L^{q_2}(X,\xi)\times L^{p_2}(X,\xi) \to L^1(G\times X, \rho^{1/{p_2}}).
\end{align}
Here, the measure on $G\times X$ is $ds\times \xi$, where $ds$ is the counting measure (i.e.~the Haar measure) on $G$. On the other hand, it is straightforward to verify that
for every $f\in L^1(X,\xi)$ and $g\in L^\infty(X,\xi)$, we have
\begin{align*} 
  \bigg|\sum_{s\in G}\int_X \mu(s)f(x)g(sx) d\xi(x)\bigg|\leq  \|f\|_{L^1(X,\xi)}\|g\|_{L^\infty(X,\xi)}.
\end{align*}
Hence, the mapping in \eqref{Eq:bilinear mapping-conv operator} defines a
bounded bilinear map
\begin{align}\label{Eq:bilinear mapping-Con operator-L1}
  T_{1,\infty}:  L^1(X,\xi)\times L^\infty(X,\xi) \to L^1(G\times X).
\end{align}
For $\theta=p_1/p_2$, we have the following complex interpolation pairs:
$$\big(L^1(X,\xi),L^{q_2}(X,\xi)\big)_\theta=L^{q_1}(X,\xi) \ , \ \big(L^\infty(X,\xi),L^{p_2}(X,\xi)\big)_\theta=L^{p_1}(X,\xi),$$
$$\big(L^1(G\times X), L^1(G\times X,\rho^{1/p_2})\big)_\theta=L^1(G\times X,\rho^{1/p_1}).$$
If we combine these facts with \eqref{Eq:bilinear mapping-Con operator-L1}, \eqref{Eq:bilinear mapping-Con operator-Lq2}, and Theorem \ref{T:n linear-interpol}, we obtain that the mapping \eqref{Eq:bilinear mapping-conv operator} defines a bounded bilinear map
\begin{align*} 
  T_{q_1,p_1}:  L^{q_1}(X,\xi)\times L^{p_1}(X,\xi) \to L^1(G\times X, \rho^{1/{p_1}}).
\end{align*}
Moreover,
$$\| T_{q_1,p_1}\|\leq \| T_{q_2,p_2}\|^{p_1/p_2},$$
or equivalently,
$$\|\pi_{q_1,X}(\mu)\| \leq  \|\pi_{q_2,X}(\mu)\|^{p_1/p_2}.$$
This implies that the mapping \eqref{Eq:asymptote Lp conv operators-(X,xi)} is increasing. Furthermore, by Proposition \ref{P:asymptote Lp Harish-Chandra function to entropy}, we have
\begin{align*}
h_\mu(X,\xi) & \geq \lim_{p\to \infty} -p\log \la \pi_{q,X}(\mu) 1, 1 \ra \\
&\geq \displaystyle\lim_{p\rightarrow \infty} -p\log \|\pi_{q,X}(\mu)\|.
\end{align*}

\end{proof}

\begin{thm}\label{T:entropy vs weak containment-RD}
Let $G$ be a countable group satisfying property RD with respect to $\fL$, and let $\mu$ be a non-degenerate probability measure on $G$ with both finite entropy and finite $\alpha$-moment with respect to $\fL$ for some $\alpha>0$. Let $(X,\xi)$ be a $(G,\mu)$-stationary space such that $\pi_{X}$ is weakly contained in $\lambda_G$. Then
\begin{align*}
\displaystyle\lim_{p\rightarrow \infty}-p\log r_{B(L^q(X,\xi))}(\pi_{q,X}(\mu))=h_\mu(X,\xi)=h(\mu,G)=\lim_{p\rightarrow \infty}-p\log r_{\pf_q(G)}(\mu).
\end{align*}
\end{thm}

\begin{proof}
Using Corollary \ref{C:charc entopry-RD groups}, we know that $G$ acts amenably on $(X,\xi)$. Hence, by \cite[Theorem 2.4]{Heb-Kuhn}, for every $\varphi\in \ell^1(G)_+$,
\begin{align}\label{Eq:RD-1}
 \|\pi_{q,X}(\varphi)\|_{B(L^q(X,\xi))}\leq \|\varphi\|_{\pf_q(G)}.
\end{align}
This, together with Proposition \ref{P:asymptote Lp conv operators-(X,xi)}, shows that for every $n\in \N$ and $p\geq 2$,
\begin{align*}
nh_\mu(X,\xi) = h_{\mu^{*n}}(X,\xi) \geq -p\log \|\pi_{q,X}(\mu^{*n})\|_{B(L^q(X,\xi))} \geq -p\log \Vert \mu^{*n} \Vert_{\pf_q(G)},
\end{align*}
so that
\begin{align*}
h_\mu(X,\xi) &\geq -p\log r_{B(L^q(X,\xi))}(\pi_{q,X}(\mu))  \\
& \geq -p\log r_{\pf_q(G)}(\mu).
\end{align*}
Hence
\begin{align*}
h_\mu(X,\xi) & \geq \displaystyle\lim_{p\rightarrow \infty}-p\log r_{B(L^q(X,\xi))}(\pi_{q,X}(\mu)) \\
& \geq \displaystyle\lim_{p\rightarrow \infty} -p\log r_{\pf_q(G)} (\mu) \\
& = h(G,\mu),
\end{align*}
where the last equality follows from Theorem \ref{T:Avez entropy-p conv operators-RD}.
This completes the proof since, by Theorem \ref{T:charc amen action-RD groups}, $h_\mu(X,\xi)\leq h(G,\mu)$.
\end{proof}

\appendix

\section{A variation on the theorem of de la Vall\'ee Poussin} \label{sec:delavalleepoussin}

This appendix is devoted to the proof of Theorem \ref{T:De LA V Poussin as used}; it is a direct consequence of the following two theorems, the first one of which is, as explained in Section \ref{subsec:vanishinglyapunovexponents}, a variation on the theorem of de la Vall\'{e}e Poussin.
\begin{thm} \label{T:De LA V Poussin}
Let $(X,\xi)$ be a probability space, and let $f\in L^1(X,\xi)$ be a non-negative function. Then there is an increasing $1$-Lipschitz function $\psi:[0,\infty)\to [0,\infty)$ such that $\lim_{y\to \infty} \psi(y)=\infty$ and such that the following holds:
If we let $\Psi:[0,\infty) \to [0,\infty)$ be the function given by
\begin{align*} 
\Psi(y)=\int_{0}^{y} \psi(z)dz \;\;\textrm{for}\;\; y\geq 0,
\end{align*}
then
\begin{align*} 
\Psi(f)\in L^1(X,\xi).
\end{align*}
\end{thm}

\begin{proof}
First, we follow the proof of the theorem of de la Vall\'ee Poussin as in \cite[Theorem 1.2]{Rao-Ren}. Let $c_1,c_2,c_3,\ldots$ be an increasing sequence of positive integers such that
\begin{align*}
  \int_{\{f > c_n\}} f \, d\xi < \frac{1}{2^n}.
\end{align*}
For $m\in \N$, let $q_m$ be the number of natural numbers $k$ such that $c_k\leq m$; that is, $q_m:=\left|\{k\in \N: c_k \leq m \} \right|$.
Since $(c_n)_{n \in \N}$ is an increasing sequence,
$$\{k\in \N: c_k \leq c_n \}=\{1,2, \ldots, n \}.$$
Hence, $q_{c_n}=n$. Moreover, for every $m\in \N$ with $c_n\leq m <c_{n+1}$, we have
$$\{k\in \N: c_k \leq m \}=\{k\in \N: c_k \leq c_n \},$$
so that
$$q_m=n \;\; \textrm{for all} \;\; m\in \N\cap [c_n,c_{n+1}).$$
Set $c_0:=0$, and let $\phi : [0,\infty)\to [0,\infty)$ be the function defined by
\begin{align*}
    \phi(y)=n \;\; \text{if} \;\; n\in \N\cup\{0\} \;\;\textrm{and}\;\;y\in [c_n,c_{n+1}).
\end{align*}
Clearly, $\phi$ is an increasing step-function on $[0,\infty)$ with $\lim_{y\to \infty} \phi(y)=\infty$. As in the proof of \cite[Theorem 1.2]{Rao-Ren}, it can be shown that if we define
\begin{align*}
\Phi(y)=\int_{0}^{y} \phi(z)dz, \;\; y\geq 0,
\end{align*}
then
\begin{align} \label{Eq:De la P-3}
    \Phi(f)\in L^1(X,\xi).
\end{align}
Now we construct our desired function $\psi : [0,\infty) \to [0,\infty)$ as follows: Let $\psi \equiv 0$ on $[0,c_1]$ and, for every $n\in \N$, let $\psi$ on $[c_n, c_{n+1}]$ be defined by requiring that on this interval, its graph corresponds to the line connecting the points $(c_n,n-1)$ and $(c_{n+1},n)$ in $\mathbb{R}^2$. It is clear that $\psi$ is continuous and increasing and that $0\leq \psi \leq \phi$ on $[0,\infty)$. Also, $\lim_{y\to \infty} \psi(y)=\infty$. If we now define $\Psi : [0,\infty) \to [0,\infty)$ by
\begin{align*}
\Psi(y)=\int_{0}^{y} \psi(z)dz, \;\; y\geq 0,
\end{align*}
then $0\leq \Psi \leq \Phi$ on $[0,\infty)$. By \eqref{Eq:De la P-3}, this implies $\Psi(f)\in L^1(X,\xi)$. Finally, note that for every $n\in \N$, the function $\psi$ is affine on $[c_n,c_{n+1}]$, with $\psi \equiv 0$ on $[0,c_1]$ and
$$\psi(c_{n+1})-\psi(c_n)=n-(n-1)=1, \;\; n\geq 1.$$
As $c_{n+1}-c_n\geq 1$ for all $n\in \N$, the function $\psi$ is $1$-Lipschitz on $[c_n,c_{n+1}]$. Since $\psi$ is increasing on $[0,\infty)$, it follows that it is $1$-Lipschitz on $[0,\infty)$.
\end{proof}

\begin{thm}\label{P:De LA P function-Subadditive}
Let $(X,\xi)$ be a probability space, and let $f\in L^1(X,\xi)$ be a non-negative function. Let $\psi$ and $\Psi$ be functions as in Theorem \ref{T:De LA V Poussin}. Then the function $F : [0,\infty) \to [0,\infty)$ defined by
\begin{align*}
    F(y)=\Psi(\log (1+y))=\int_{0}^{\log(1+y)} \psi(z)dz, \;\; y\geq 0,
\end{align*}
is increasing and differentiable on $[0,\infty)$. Moreover, there exists $M> 0$ such that
\begin{align*}
  F(y+y') \leq F(y)+F(y')+M \;\;\textrm{for all}\;\; y,y'\geq 0.
\end{align*}
\end{thm}

\begin{proof}
Since $\psi$ is continuous and increasing on $[0,\infty)$, the function $F$ is continuously differentiable and increasing on $[0,\infty)$, and we have
\begin{align*}
  F'(y)=\frac{\psi(\log(1+y))}{1+y} \;\; \textrm{for all} \;\; y\geq 0.
\end{align*}
We claim that there exists $c>0$ such that $F'$ is decreasing on $[c,\infty)$.
In order to show this, by replacing $y$ with $e^u-1$, it suffices to show that the function
$\Lambda : [0,\infty)\to [0,\infty)$ given by
\begin{align*}
  \Lambda(u):=F'(e^u-1)=\frac{\psi(u)}{e^u}
\end{align*}
is decreasing on $[c',\infty)$ for some $c'>0$. Since $\lim_{u \to \infty} \psi(u) = \infty$, there exists $c'>0$ such that $u\geq c'$ implies $\psi(u)\geq 1$.
Now let $c' \leq u \leq w$. Since $\psi$ is increasing and $1$-Lipschitz, we have
$$\psi(w) \leq \psi(u)+w-u,$$
and therefore,
\begin{align*}
\frac{\psi(w)}{e^w} & \leq \frac{\psi(u)+(w-u)}{e^ue^{w-u}} \leq \frac{\psi(u)e^{w-u}}{e^ue^{w-u}}=\frac{\psi(u)}{e^u},
\end{align*}
where the second inequality uses $\psi(u)\geq 1$ and $w-u \geq 0$. Hence, $\Lambda$ is decreasing on $[c',\infty)$, so $F'$ is decreasing on $[c,\infty)$ for some $c > 0$.

This has two implications: First, by a standard subadditivity argument, it follows that the function
$$[0,\infty)\to [0,\infty), \;\; y\mapsto F(y+c)-F(c)$$
is subadditive so that, as $F$ is increasing,
\begin{align}\label{Eq:subadditive-1}
F(y+y')\leq F(y+y'+c)\leq F(y+c)+F(y'+c)-F(c), \;\; y,y'\geq 0.
\end{align}
Secondly, as $F'$ is decreasing on $[c,\infty)$ and continuous on $[0,\infty)$, it is bounded on $[0,\infty)$. Thus, there exists $M_1>0$  such that
\begin{align*} 
F'(y)\leq M_1 \;\;\textrm{for all}\;\; y\geq 0.
\end{align*}
From the mean value theorem, it follows that
\begin{align}\label{Eq:subadditive-2}
F(y+c)-F(y) \leq c M_1 \;\;\textrm{for all}\;\; y\geq 0.
\end{align}
Combining \eqref{Eq:subadditive-1} and \eqref{Eq:subadditive-2}, it follows that for every $y,y'\geq 0$,
$$F(y+y') \leq F(y)+F(y')+2cM_1-F(c).$$
The final result now follows by setting $M:=\max\{1,2cM_1-F(c)\}.$
\end{proof}


\begin{thebibliography}{99}

\bibitem{Adam-Elliot-Gior} S.~Adams, G.A.~Elliott and T.~Giordano, \emph{Amenable actions of groups},
Trans. Amer. Math. Soc. \textbf{344} (1994), 803--822.

\bibitem{Ana-Dela 1} C.~Anantharaman-Delaroche, \emph{On spectral characterizations of amenability}, Israel J. Math. \textbf{137} (2003), 1--33.

\bibitem{Avez} A.~Avez, \emph{Harmonic functions on groups}, in: Differential geometry and relativity, pp. 27--32, Reidel, Dordrecht-Boston, Mass., 1976.

\bibitem{badershalom} U.~Bader and Y.~Shalom, \emph{Factor and normal subgroup theorems for lattices in products of groups}, Invent. Math. \textbf{163} (2006), 415--454.

\bibitem{BL} J.~Bergh and J.~L\"ofstr\"om, \emph{Interpolation spaces. An introduction}, Springer-Verlag, Berlin-New York, 1976.

\bibitem{boutonnethoudayer} R.~Boutonnet and C.~Houdayer, \emph{Stationary characters on lattices of semisimple Lie groups}, Publ. Math. Inst. Hautes Études Sci. \textbf{133} (2021), 1--46.

\bibitem{Cal} A.-P.~Calder\'on, \emph{Intermediate spaces and interpolation, the complex method}, Studia Math. \textbf{24} (1964), 113--190.

\bibitem{Chatterji2} I.~Chatterji, \emph{Property (RD) for cocompact lattices in a finite product of rank one Lie groups with some rank two Lie groups}, Geom. Dedicata \textbf{96} (2003), 161--177.

\bibitem{Chatt 1} I.~Chatterji, \emph{Introduction to the rapid decay property}, in: Around Langlands correspondences, pp.~53--72, Contemp. Math., Vol. 691, Amer. Math. Soc., Providence, RI, 2017.

\bibitem{CR} I.~Chatterji and K.~Ruane, \emph{Some geometric groups with rapid decay}, Geom. Funct. Anal. \textbf{15} (2005), 311--339.

\bibitem{Derriennic} Y.~Derriennic, \emph{Quelques applications du théorème ergodique sous-additif}, in:~Conference on Random Walks (Kleebach, 1979), pp.~183--201, 4, Astérisque, Vol.~74, Soc. Math. France, Paris, 1980.

\bibitem{DS} C.~Dru\c{t}u and M.~Sapir, \emph{Relatively hyperbolic groups with rapid decay property}, Int. Math. Res. Not. (2005), 1181--1194.

\bibitem{Folland} G. B. Folland, \emph{Real analysis: Modern techniques and their applications}, 2nd edition, John Wiley \& Sons, Inc., New York, 1999.

\bibitem{FHTV} J.~Frisch, Y.~Hartman, O.~Tamuz, P.~Vahidi Ferdowsi, \emph{Choquet--Deny groups and the infinite conjugacy class property}, Ann. of Math. (2) \textbf{190} (2019), 307--320.

\bibitem{FTV} J.~Frisch, O.~Tamuz and P. Vahidi Ferdowsi, \emph{Strong amenability and the infinite conjugacy class property}, Invent. Math. \textbf{218} (2019), 833--851.

\bibitem{Furm1} A.~Furman, \emph{Random walks on groups and random transformations}, in:~Handbook of dynamical systems, Vol.~1A, pp.~931--1014, North-Holland, Amsterdam, 2002.

\bibitem{Furstenberg1963} H.~Furstenberg, \emph{A Poisson formula for semi-simple Lie groups}, Ann. of Math. (2) \textbf{77} (1963), 335--386.

\bibitem{FurstenbergNCRP} H.~Furstenberg, \emph{Noncommuting random products}, Trans. Amer. Math. Soc. \textbf{108} (1963), 377--428.

\bibitem{Furstenberg1973} H.~Furstenberg, \emph{Boundary theory and stochastic processes on homogeneous spaces}, in: Harmonic analysis on homogeneous spaces (Proc. Sympos. Pure Math., Vol. XXVI, Williams Coll., Williamstown, Mass., 1972), pp. 193--229, Amer. Math. Soc., Providence, RI, 1973.

\bibitem{Furs-Glas 1} H.~Furstenberg and E.~Glasner, \emph{Stationary dynamical systems}, in: Dynamical
numbers -- interplay between dynamical systems and number theory, pp. 1--28, Contemp. Math., Vol. 532, Amer. Math. Soc., Providence, RI, 2010.

\bibitem{Garncarek} \L.~Garncarek, \emph{Mini-course: property of rapid decay}, preprint (2016), arXiv:1603.06730.

\bibitem{G1} Y. Guivarc'h, \emph{Sur la loi des grands nombres et le rayon spectral d'une marche al\'{e}atoire},
in:~Conference on Random Walks (Kleebach, 1979), pp.~47--98, 3, Astérisque, Vol.~74, Soc. Math. France, Paris, 1980.

\bibitem{Haagerup} U.~Haagerup, \emph{An example of a nonnuclear $C^*$-algebra, which has the metric approximation property}, Invent. Math. \textbf{50} (1978/79), 279--293.

\bibitem{Hart-Kal 2} Y.~Hartman and M.~Kalantar, \emph{Stationary $C^*$-dynamical systems},
J. Eur. Math. Soc. (JEMS) \textbf{25} (2023), 1783--1821.

\bibitem{Heb-Kuhn} W.~Hebisch and G.M.~Kuhn, \emph{Transference for amenable actions}, Proc. Am. Math. Soc. \textbf{133}, (2005), 1733--1740.

\bibitem{Herz} C.~Herz, \emph{Harmonic synthesis for subgroups},
Ann. Inst. Fourier (Grenoble) \textbf{23} (1973), 91--123.

\bibitem{Joli} P.~Jolissaint, \emph{Rapidly decreasing functions in reduced C$^*$-algebras of groups}, Trans. Amer. Math. Soc. \textbf{317} (1990), 167--196.

\bibitem{Kaim 1} V.A.~Kaimanovich, \emph{The Poisson formula for groups with hyperbolic properties}, Ann. of Math. (2) \textbf{152} (2000), 659--692.

\bibitem{Kaim-Vers 1} V.A.~Kaimanovich and A.~Vershik, \emph{Random walks on discrete groups: boundary and entropy}, Ann. Probab. \textbf{11} (1983), 457--490.

\bibitem{Kesten} H.~Kesten, \emph{Full Banach mean values on countable groups}, Math. Scand. \textbf{7} (1959), 146--156.	

\bibitem{Kuznets 1} Yu.N.~Kuznetsova, \emph{The invariant weighted algebras $\mathcal{L}^\om_p(G)$}, Mat. Zametki \textbf{84} (2008), 567--576; English translation: Math. Notes \textbf{84} (2008), 529--537.

\bibitem{Lafforgue}
V.~Lafforgue, \emph{A proof of property (RD) for cocompact lattices of $\mathrm{SL}(3,\mathbf{R})$ and $\mathrm{SL}(3,\mathbf{C})$}, J. Lie Theory \textbf{10} (2000), 255--267.

\bibitem{Ledrappier1982}
F.~Ledrappier, \emph{Quelques propriétes des exposants charactéristiques}, in: Lecture Notes in Math., Vol. 1097, pp.~306--396, Springer, New York, 1982.

\bibitem{LMR} A.~Lubotzky, S.~Mozes and M.S.~Raghunathan, \emph{Cyclic subgroups of exponential growth and metrics on discrete groups}, C. R. Acad. Sci. Paris Sér. I Math. \textbf{317} (1993), 735--740.

\bibitem{Lyons-Peres} R.~Lyons and Y.~Peres, \emph{Probability on trees and networks}, Cambridge University Press, New York, 2016.

\bibitem{mackay} D.J.C.~MacKay, \emph{Information theory, inference and learning algorithms}, Cambridge University Press, New York, 2003.

\bibitem{Marg1} G.A. Margulis, \emph{Discrete subgroups of semisimple Lie groups}, Springer, New York, 1991.

\bibitem{Nevo1} A. Nevo, \emph{The spectral theory of amenable actions and invariants of discrete groups}, Geom. Dedicata \textbf{100} (2003), 187--218.

\bibitem{nevosageev} A.~Nevo and M.~Sageev, \emph{The Poisson boundary of $\mathrm{CAT}(0)$ cube complex groups}, Groups Geom. Dyn. \textbf{7} (2013), 653--695.

\bibitem{NevoZimmer} A.~Nevo and R.J.~Zimmer, \emph{Rigidity of Furstenberg entropy for semisimple Lie group actions}, Ann. Sci. École Norm. Sup. (4) \textbf{33} (2000), 321--343.

\bibitem{Pytlik} T. Pytlik, \emph{Symbolic calculus on weighted group algebras}, Studia Math. \textbf{73} (1982), 169--176.

\bibitem{RS} H.~Reiter and J.D.~Stegeman, \emph{Classical harmonic analysis and locally compact groups}, Oxford University Press, 2000.

\bibitem{Rammage-Robertson-Steger} J.~Ramagge, G.~Robertson and T.~Steger, \emph{A Haagerup inequality for $\widetilde{A}_1 \times \widetilde{A}_1$ and $\widetilde{A}_2$ buildings}, Geom. Funct. Anal. \textbf{8} (1998), 702--731.

\bibitem{Rao-Ren} M.M.~Rao and Z.D.~Ren, \emph{Theory of Orlicz spaces}, Marcel Dekker, Inc., New York, 1991.

\bibitem{SW-exotic} E.~Samei and M.~Wiersma, \emph{Exotic $C^*$-algebras of geometric groups}, J. Funct. Anal. \textbf{286} (2024), Paper No. 110228, 32 pp.

\bibitem{Valette} A.~Valette, \emph{Introduction to the Baum-Connes conjecture}, Birkhäuser Verlag, Basel, 2002.

\bibitem{watrous} J.~Watrous, \emph{The theory of quantum information}, Cambridge University Press, Cambridge, 2018.

\bibitem{Wilkinson2017} A.~Wilkinson, \emph{What are Lyapunov exponents, and why are they interesting?}, Bull. Amer. Math. Soc. (N.S.) \textbf{54} (2017), 79--105.

\bibitem{woess} W.~Woess, \emph{Random walks on infinite graphs and groups}, Cambridge University Press, Cambridge, 2000.

\bibitem{Zheng} T.~Zheng, \emph{Asymptotic behaviors of random walks on countable groups}, ICM -- International Congress of Mathematicians, Vol. 4, pp.~3340--3365, EMS Press, Berlin, 2023.

\bibitem{Zimm 2} R.J.~Zimmer, \emph{Amenable ergodic group actions and an application to Poisson boundaries of random walks}, J. Funct. Anal.  \textbf{27} (1978), 350--372.

\end{thebibliography}
\end{document}